\documentclass[fullpage]{amsart}

\usepackage[a4paper]{geometry}
\DeclareMathOperator{\de}{d}

\usepackage{hyperref} 

\usepackage{amsmath, amssymb} 
\usepackage{latexsym}

\usepackage{bbm}
\usepackage{enumitem}
\setenumerate{label={\normalfont(\alph*)},topsep=4pt,itemsep=0pt} 

\usepackage[UKenglish]{babel}

\usepackage{graphicx}
\usepackage{tikz}

\usetikzlibrary{calc}

\usepackage{subfiles}
\usepackage{color}

\usetikzlibrary{arrows,shapes,automata,positioning}

\definecolor{brightpink}{rgb}{1.0, 0.0, 0.5}
\definecolor{capri}{rgb}{0.0, 0.75, 1.0}

\usepackage{mathtools}

\DeclarePairedDelimiter\abs{\lvert}{\rvert}
\DeclarePairedDelimiter\norm{\lVert}{\rVert}
\makeatletter
\let\oldabs\abs
\def\abs{\@ifstar{\oldabs}{\oldabs*}}
\let\oldnorm\norm
\def\norm{\@ifstar{\oldnorm}{\oldnorm*}}
\makeatother

\usepackage{caption}
\usepackage{subfigure}

\newtheorem{theorem}{Theorem}[section]
\newtheorem{lemma}[theorem]{Lemma}
\newtheorem{corollary}[theorem]{Corollary}
\newtheorem{definition}[theorem]{Definition}
\newtheorem{example}[theorem]{Example}
\newtheorem{proposition}[theorem]{Proposition}
\newtheorem{remark}[theorem]{Remark}
\newtheorem{main}{Theorem}

\begin{document}
\title{On approximation by random L\"uroth expansions}
\author[Charlene Kalle]{Charlene Kalle}
\address[Charlene Kalle]{Mathematisch Instituut, Leiden University, Niels Bohrweg 1, 2333CA Leiden, The Netherlands}
\email[Charlene Kalle]{kallecccj@math.leidenuniv.nl}
\author[Marta Maggioni]{Marta Maggioni$^\dagger$}
\address[Marta Maggioni]{Mathematisch Instituut, Leiden University, Niels Bohrweg 1, 2333CA Leiden, The Netherlands}
\email[Marta Maggioni]{m.maggioni@math.leidenuniv.nl}

\thanks{$^\dagger$ The second author was supported by the NWO TOP-Grant No.~614.001.509.}

\begin{abstract}
We introduce a family of random $c$-L\"uroth transformations $\{L_c\}_{c \in [0, \frac12]}$, obtained by randomly combining the standard and alternating L\"uroth maps with probabilities $p$ and $1-p$, $0 < p < 1$, both defined on the interval $[c,1]$. We prove that the pseudo-skew product map $L_c$ produces for each $c \le \frac25$ and for Lebesgue almost all $x \in [c,1]$ uncountably many different generalised L\"uroth expansions that can be investigated simultaneously. Moreover, for $c= \frac1{\ell}$, for $\ell \in \mathbb{N}_{\geq 3} \cup \{\infty\}$, Lebesgue almost all $x$ have uncountably many universal generalised L\"uroth expansions with digits less than or equal to $\ell$. For $c=0$ we show that typically the speed of convergence to an irrational number $x$, of the sequence of L\"uroth approximants generated by $L_0$, is equal to that of the standard L\"uroth approximants; and that the quality of the approximation coefficients depends on $p$ and varies continuously between the values for the alternating and the standard L\"uroth map. Furthermore, we show that for each $c \in \mathbb Q$ the map $L_c$ admits a Markov partition. For specific values of $c>0$, we compute the density of the stationary measure and we use it to study the typical speed of convergence of the approximants and the digit frequencies.
\end{abstract}

\maketitle

\section{Introduction}
In 1883 L\"uroth showed in \cite{Lu} that each $x \in [0,1]$ can be expressed in the form
\begin{equation}\label{d:Lexp}
x = \frac{1}{d_1} + \frac{1}{d_1(d_1-1)d_2} +  \ldots =  \sum_{m \ge 1} (d_m-1) \prod_{j=1}^m \frac1{d_j(d_j-1)},
\end{equation}
where $d_m  \in \mathbb N_{\ge 2} \cup \{ \infty\}$ for each $m$ (and with $\frac1{\infty}=0$). Such expressions are now called {\em L\"uroth expansions}. By considering the numbers
\begin{equation}\label{q:sequence}
\frac{p_n}{q_n} := \sum_{m = 1}^n (d_m-1) \prod_{j=1}^m \frac1{d_j(d_j-1)}, \quad n \ge 1,
\end{equation}
one obtains a sequence of rationals converging to the number $x$, making L\"uroth expansions suitable for finding rational approximations of irrational numbers. Since their introduction in 1883 much research has been done on the approximation properties of L\"uroth expansions from various perspectives. In this article we address these questions by adopting a random dynamical systems approach. It turns out that this yields for each $x$ many different number expansions similar to the L\"uroth expansion from \eqref{d:Lexp}, without compromising the quality of approximation. Before we state our results, we first give a brief summary of a selection of the known results.\\

\vskip .2cm
A L\"uroth expansion is called {\em ultimately periodic} if there exist $n \ge 0$ and $r \ge 1$ such that $d_{n+j} = d_{n+r+j}$ for all $j \ge 1$ (and {\em periodic} if $n=0$). One of the most basic results on L\"uroth expansions, obtained in \cite{Lu}, is on periodicity.

\begin{theorem}{\cite[page 416]{Lu}}\label{t:lurothperiodic}
A real number $x \in (0,1)$ has an ultimately periodic L\"uroth expansion if and only if $x \in \mathbb Q$.
\end{theorem}
%In 1883 L\"uroth showed in \cite{Lu} that each $x \in (0,1)$ can be expressed in the form
%\begin{equation}\label{d:Lexp}
%x = \frac{1}{d_1} + \frac{1}{d_1(d_1-1)d_2} +  \ldots =  \sum_{n \ge 1} (d_n-1) \prod_{j=1}^n \frac1{d_j(d_j-1)},
%\end{equation}
%where $d_n  \in \mathbb N_{\ge 2} \cup \{ \infty\}$ for each $n$. Such expressions are now called {\em L\"uroth expansions}.
Many other properties of L\"uroth expansions were obtained using a dynamical system. Indeed, L\"uroth expansions can be obtained dynamically by iterating the {\em L\"uroth transformation} $T_L:[0,1] \to [0,1]$ given by $T_L(0)=0$, $T_L(1)=1$ and
\[T_L(x) = \bigg\lceil  \frac{1}{x} \bigg\rceil \bigg( \bigg\lceil  \frac{1}{x} \bigg\rceil -1 \bigg) x - \bigg( \bigg\lceil \frac{1}{x}  \bigg\rceil-1 \bigg)\]
for $x \neq 0,1$, where $\lceil x \rceil$ denotes the smallest integer not less than $x$. See Figure \ref{f:luroth2}(a) for the graph.\\

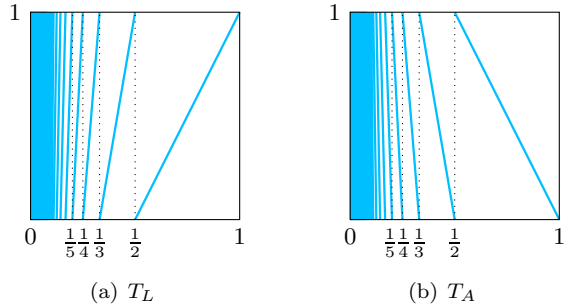
\begin{figure}[h!]
 \centering
 \subfigure[$T_L$]{
 \begin{tikzpicture}[scale=2.75]
 \draw[fill=capri, capri]  (0,0) -- (.077,0) -- (.077,1) -- (0,1) -- cycle;
\draw[line width=0.3mm, capri] (.077,0)--(.083,1)(.083,0)--(.09,1)(.09,0)--(.1,1)(.1,0)--(.11,1)(.11,0)--(.125,1)(.125,0)--(.143,1)(.143,0)--(.167,1)(.167,0)--(.2,1)(.2,0)--(.25,1)(.25,0)--(.33,1)(.33,0)--(.5,1)(.5,0)--(1,1);
\draw(0,0)node[below]{\small $0$}--(.19,0)node[below]{\small $\frac15$}--(.255,0)node[below]{\small $\frac14$}--(.33,0)node[below]{\small $\frac13$}--(.5,0)node[below]{\small $\frac12$}--(1,0)node[below]{\small $1$}--(1,1)--(0,1)node[left]{\small $1$}--(0,0);
\draw[dotted](.2,0)--(.2,1)(.25,0)--(.25,1)(.33,0)--(.33,1)(.5,0)--(.5,1);
\end{tikzpicture}}
\hspace{.5cm}
 \subfigure[$T_A$]{
 \begin{tikzpicture}[scale=2.75]
 \draw[fill=capri, capri]  (0,0) -- (.077,0) -- (.077,1) -- (0,1) -- cycle;
\draw[line width=0.3mm, capri]  (.077,1)--(.083,0)(.083,1)--(.09,0)(.09,1)--(.1,0)(.1,1)--(.11,0)(.11,1)--(.125,0)(.125,1)--(.143,0)(.143,1)--(.167,0)(.167,1)--(.2,0)(.2,1)--(.25,0)(.25,1)--(.33,0)(.33,1)--(.5,0)(.5,1)--(1,0);
\draw(0,0)node[below]{\small $0$}--(.19,0)node[below]{\small $\frac15$}--(.255,0)node[below]{\small $\frac14$}--(.33,0)node[below]{\small $\frac13$}--(.5,0)node[below]{\small $\frac12$}--(1,0)node[below]{\small $1$}--(1,1)--(0,1)node[left]{\small $1$}--(0,0);
\draw[dotted](.2,0)--(.2,1)(.25,0)--(.25,1)(.33,0)--(.33,1)(.5,0)--(.5,1);
\end{tikzpicture}}
\caption{The standard and the alternating L\"uroth maps in (a) and (b), respectively.}
\label{f:luroth2}
\end{figure}

\vskip .2cm
The digits $d_n$, $n \ge 1$, are obtained by setting $d_n (x) = k$ if $T_L^{n-1}(x) \in \big[ \frac1k, \frac1{k-1}\big)$, $k \ge 2$, and $d_n(x)=\infty$ if $T_L^{n-1}(x)=0$. Hence, the map $T_L$ produces for each $x \in [0,1]$ a L\"uroth expansion as in \eqref{d:Lexp}. From the graph of $T_L$ one sees immediately that Lebesgue almost all numbers $x \in [0,1]$ have a unique L\"uroth expansion and if $x$ does not have a unique expansion, then it has exactly two different ones, one with $d_n =d$ and $d_{n+j} = \infty$ and one with $d_n=d+1$ and $d_{n+j}=2$ for some $n,d$ and all $j \ge 1$. This holds for any number $x \in [0,1]$ for which there is an $n\ge1$ such that $T_L^n(x) =0$. By identifying these two expansions we can speak of the unique L\"uroth expansion of any number $x \in [0,1]$.\\

\vskip .2cm
From the dynamics of $T_L$ we get information on the digit frequencies in L\"uroth expansions. The map $T_L$ is measure preserving and ergodic with respect to the Lebesgue measure $\lambda$ on $[0,1]$. It is then a straightforward application of Birkhoff's Ergodic Theorem that, in the L\"uroth expansion of Lebesgue almost every $x$, the frequency of the digit $d$ equals $\frac1{d(d-1)}$, corresponding to the length of the interval $\big[ \frac1d, \frac1{d-1} \big)$. It was proven by \v{S}al\'at in \cite{Sal} that for any $D \ge 2$ the set of points $x \in (0,1)$ for which all L\"uroth expansion digits are bounded by $D$ has Hausdorff dimension $<1$ with the dimension approaching 1 as $D \to \infty$. The articles \cite{Sal,BaBu,BI09,SF11,CWZ13,MT13,SFM17} all consider L\"uroth expansions with certain restrictions on the digits $d_n$.

\vskip .2cm
The quality of approximation by L\"uroth expansions depends on the {\em approximants} or {\em convergents}\index{convergents} $\frac{p_n}{ q_n}$ given in \eqref{q:sequence}. In \cite{BI09} the authors give a multifractal analysis of the speed with which the sequence $\big( \frac{p_n}{ q_n}\big)_n$ converges to the corresponding $x$ using the Lyapunov exponent. The Lyapunov exponent of $T_L$ at $x \in (0,1)$ is defined by
\[ \Lambda_L(x) = \lim_{n \to \infty} \frac1n \log \sum_{k=0}^{n-1} |T_L' (T_L^k (x))|,\]
whenever this limit exists. It follows from another application of Birkhoff's Ergodic Theorem that for $\lambda$-a.e.~$x \in (0,1)$,
\[ \Lambda(x) = \sum_{d=2}^\infty \frac{\log(d(d-1))}{d(d-1)}.\]
The authors of \cite{BI09} obtain, among other things, the following result.
\begin{theorem}[\cite{BI09}]\label{t:bi}
For $\lambda$-a.e.~$x \in (0,1)$,
\[ \lim_{n \to \infty} \frac1n \log \Big| x - \frac{p_n}{q_n} \Big| = - \sum_{d=2}^\infty \frac{\log(d(d-1))}{d(d-1)}.\]
Moreover, the range of possible values of this rate is $(-\infty, -\log 2]$.% the Hausdorff dimension of the set of points $x \in (0,1)$ that have $\Lambda_L(x)=\gamma$ is given by
%\[ \frac1{\gamma} \int_{t \in \mathbb R} \Big( P(-t\log |T_L'|) + t\gamma \big).\]
\end{theorem}

%From \eqref{d:Lexp} one sees that the size of the digits $d_n$ influences the distance from $\frac{p_n}{q_n}$ to $x$. It is another straightforward application of the Birkhoff Ergodic Theorem that the digit $d\ge 2$ occurs with frequency $\frac1{d(d-1)}$ in the L\"uroth expansion of a typical point $x$. 

\vskip .2cm
Another way to express the quality of the approximations is via the limiting behaviour of the {\em approximation coefficients}\index{approximation coefficients}
\begin{equation}\label{q:appcf}
\theta_n^L(x) := q_n \Big| x - \frac{p_n}{q_n}\Big|,
\end{equation}
where $q_n = d_n \prod_{i=1}^{n-1} d_i(d_i-1)$. In \cite{DK96} the authors prove the following result.
\begin{theorem}[Theorem 2, \cite{DK96}]
For $\lambda$-a.e.~$x\in [0,1]$ and for every $z \in (0,1]$ the limit
\[ \lim_{N \to \infty} \frac{\# \{ 1 \le j \le N \, : \, \theta_j^L(x) < z\}}{N}\]
exists and equals
\begin{equation}\label{q:fl}
F_L(z) := \sum_{k=2}^{\lfloor \frac1z \rfloor +1} \frac{z}{k} + \frac1{\lfloor \frac1z \rfloor +1}.
\end{equation}
\end{theorem}
We refer to e.g.~\cite{SYZ14,Van14,Giu16,GL16,ZC16,She17,LCTW18,SX18,TW18} for results on other properties of L\"uroth expansions.\\

\vskip .2cm
In \cite{BaBu} the authors placed the map $T_L$ in the larger framework of {\em generalised L\"uroth series transformations} (GLS). A GLS transformation is a piecewise affine onto map $T_{\mathcal P,\varepsilon}:[0,1]\to [0,1]$ given by an at most countable interval partition $\mathcal P$ of $[0,1]$ and a vector $\varepsilon = (\varepsilon_n)_n \in \{ 0,1\}^{\# \mathcal P}$ specifying for each partition element the orientation of $T_{\mathcal P, \varepsilon}$ on that interval. The L\"uroth transformation can be obtained by taking the partition $ \mathcal P_L = \big\{ \big[ \frac1{n}, \frac{1}{n-1} \big) \big\}_{n \ge 2}$ and orientation vector $\varepsilon = (0)_{n \ge 1}$, i.e., all branches are orientation preserving. In \cite{BaBu} the authors considered all GLS transformations $T_{\mathcal P_L, \varepsilon}$ with partition $\mathcal P_L$. Besides the L\"uroth transformation, another specific instance of this family is the {\em alternating L\"uroth map} $T_A: [0,1]\to [0,1]$ given by $T_A(x) = 1-T_L(x)$, see Figure~\ref{f:luroth2}(b), which has $\varepsilon = (1)_{n \ge 1}$, so that all branches orientation reversing. Similar to the L\"uroth expansion from \eqref{d:Lexp}, iterations of any GLS transformation $T_{\mathcal P_L, \varepsilon}$ yield number expansions for $x\in [0,1]$ of the form
\begin{equation}\label{d:glse}
x=\sum_{n=1}^{\infty} (-1)^{\sum_{i=1}^{n-1} s_i} \frac{d_n-1 + s_n}{\prod_{i=1}^{n} d_i(d_i-1)},
\end{equation}
where $s_n \in \{0,1\}$ and $d_n \ge 2$, called {\em generalised L\"uroth expansions}. Here we let $\sum_{i=1}^0 s_i=0$. For each map $T_{\mathcal P_L,\varepsilon}$ the authors of \cite{BaBu} consider the approximation coefficients $\theta_n^{\mathcal P_L,\varepsilon}$ and the corresponding distribution function $F_{\mathcal P_L,\varepsilon}$ and they find the following.

\begin{theorem}[Theorem 4, \cite{BaBu}]\label{t:fafl}
The distribution function of $\theta_n^A$ for the map $T_A$ is given for $0 < z \leq 1$ by
\[ F_A(z) = \sum_{k=2}^{\lfloor \frac1z \rfloor} \frac{z}{k-1} + \frac1{\lfloor \frac1z \rfloor}.\]
For any GLS transformation $T_{\mathcal{P}_L,\varepsilon}$ it holds that
\[F_A \leq F_{\mathcal P_L,\varepsilon} \leq  F_L,\]
Furthermore, the first moments of $F_L$ and $F_A$ are given respectively by 
\[M_L := \int_{[0,1]} 1-F_L \, d\lambda = \frac{ \zeta(2)}{2}- \frac12 \quad \text{and} \quad M_A := \int_{[0,1]} 1-F_A \, d\lambda = 1- \frac{ \zeta(2)}{2},\]
where $\zeta(2)$ is the zeta function evaluated at $2$.
\end{theorem}
The authors of \cite{BaBu} remark that they suspect that the set of values that the limit $ \lim_{n \to \infty} \frac1n \sum_{i=1}^n \theta_i^{\mathcal P_L, \varepsilon}$ can take is a fractal set inside the interval $[M_A, M_L]$. Other results on the map $T_A$ can be found e.g.~in \cite{KaKn1,KaKn2}. For results on different families of GLS transformations, see e.g.~\cite{KMS11,Mun11,CWY14,CW14}.\\

\vskip .2cm
In this article we adopt a random approach to L\"uroth expansions. We introduce a family of random L\"uroth systems $\{L_{c,p}\}_{c \in [0, \frac12], 0 \leq p \leq 1}$ that are obtained from randomly combining the maps $T_L$ and $T_A$. The parameter $c$ is the cutting point, that defines the interval $[c,1]$ on which each $L_{c,p}$ is defined. More precisely, we overlap $T_L$ and $T_A$ on the interval $[c,1]$ and remove from both maps the pieces that map points into $[0,c)$. The parameter $p$ reflects the probability with which we apply the map $T_L$. To be precise, let $T_0:= T_L$ and $T_1 := T_A$ and let $\sigma$ denote the left shift on sequences. Then the {\em random $c$-L\"uroth transformation} $L_{c,p}: \{0,1\}^\mathbb N \times [c,1] \to \{0,1\}^\mathbb N \times [c,1]$ is defined by
\[ L_{c,p} (\omega,x) = \big( \sigma(\omega), T_{\omega_1}(x) 1_{[c,1]} (T_{\omega_1}(x)) + T_{1-\omega_1} (x) 1_{[0,c)}(T_{\omega_1}(x)) \big).\]
By iteration, each map $L_{c,p}$ produces for each pair $(\omega,x)$ a generalised L\"uroth expansion for $x$ as in \eqref{d:glse}. So for typical $x \in [c,1]$ multiple generalised L\"uroth expansions of the form \eqref{d:glse} are obtained. If $c=0$ the corresponding random L\"uroth expansions have digits in the set $\mathbb N_{\ge 2} \cup \{ \infty\}$, but for $c>0$ the available set of digits is bounded from above. This makes the two cases inherently different. We summarise our main results in the following three theorems.

\begin{main}\label{t:main1} Let $c \in \big[0, \frac12 \big]$.
\begin{itemize}
\item[(i)] If $x \in [c,1] \setminus \mathbb Q$, then no generalised L\"uroth expansion of $x$ produced by $L_{c,p}$ is ultimately periodic.
\item[(ii)] If $x \in [c,1] \cap \mathbb Q$ then, depending on the values of $x$ and $c$, the map $L_{c,p}$ can produce any of the following number of different generalised L\"uroth expansions:
\begin{itemize}
\item a unique expansion,
\item a finite or countable number of ultimately periodic expansions,
\item countably many ultimately periodic expansions and uncountably many expansions that are not ultimately periodic.
\end{itemize}
\end{itemize}
\end{main}
We also give a characterisation on when each of these cases occurs. This result does not depend on the value of $p$. The following results further explore the properties of the produced generalised L\"uroth expansions in case $c=0$ and $c>0$. We call a generalised L\"uroth expansion generated by a map $L_{c,p}$ {\em universal} if any possible block of digits of any length from the alphabet associated to $L_{c,p}$ occurs in the expansion.

\begin{main}\label{t:main2}
Let $c=0$ and $0 < p < 1$.
\begin{itemize}
\item[(i)] The map $L_{0,p}$ generates for Lebesgue almost every $x \in [0,1]$ uncountably many universal generalised L\"uroth expansions.
\item[(ii)] The speed of convergence of the sequence $\big(\frac{p_n}{q_n}\big)_n$ to $x$ for any generalised L\"uroth expansion produced by $L_{0,p}$ typically satisfies
\[ \lim_{n \to \infty} \frac1n \log \Big| x - \frac{p_n}{q_n} \Big| = - \sum_{d=2}^\infty \frac{\log(d(d-1))}{d(d-1)}\]
and the range of possible values of this rate is $(-\infty, -\log 2]$. In particular, this rate does not depend on $p$. 
\item[(iii)] Typically the approximation coefficients generated by $L_{0,p}$ satisfy
\[ \lim_{n \to \infty} \frac1n \sum_{i=1}^n \theta_i^{0,p} = p \frac{2\zeta(2)-3}{2} + \frac{2-\zeta(2)}{2},\]
where $\zeta(2)$ is the zeta function evaluated at 2. In particular, this limit can attain any value in the interval $[M_A, M_L]$.
\end{itemize}
\end{main}
The result in (ii) is given by considering the Lyapunov exponent of the random system $L_{0,p}$ as was done for Theorem~\ref{t:bi}. We see that the speed of convergence is not compromised by adding randomness to the system. For (iii) we note that instead of a fractal set inside $[M_A, M_L]$ we can obtain the full interval by adding randomness. 

\begin{main}\label{t:main3}
Let $c>0$ and $0 < p < 1$.
\begin{itemize}
\item[(i)] If $0< c \le \frac25$, then the map $L_{c,p}$ generates for every irrational $x \in [c,1]$ uncountably many different generalised L\"uroth expansions.
\item[(ii)] If $c = \frac1{\ell}$ for some $\ell \in \mathbb N_{\ge 3}$, then $L_{c,p}$ generates for every irrational $x \in [c,1]$ uncountably many universal generalised L\"uroth expansions.
\end{itemize}
\end{main}

Notice that results corresponding to (ii) and (iii) from Theorem~\ref{t:main2} in case $c>0$ are missing. For $c>0$ the speed of convergence of the sequence $\big(\frac{p_n}{q_n} \big)_n$ is still governed by the Lyapunov exponent of the map $L_{c,p}$, but this is not so easily computed. For $c=0$ we are in the lucky circumstance that $m_p \times \lambda$ is an invariant measure for $L_{0,p}$, where $m_p$ is the $(p,1-p)$-Bernoulli measure on $\{0,1\}^\mathbb N$  and $\lambda$ is the one-dimensional Lebesgue measure. General results give the existence of an invariant measure of the form $m_p \times \mu_{p,c}$ where $\mu_{p,c} \ll \lambda$. In most cases the random systems $L_{c,p}$ satisfy the conditions from \cite[Theorem 4.1]{KM}, which gives an expression for the density $\frac{\de \mu_{p,c}}{\de \lambda}$ in individual cases. In the last section we discuss some values of $c$ for which we can determine a nice formula for this density. We then get a result similar to Theorem~\ref{t:main2}(ii) and compute the frequency of the digits $d$ in the generalised L\"uroth expansions. 
%Finding results corresponding to (ii) and (iii) from Theorem~\ref{t:main2} in case $c>0$ requires a good knowledge of the invariant measures of the random systems $L_{c,p}$. Rather, we prove that for each $c \in \big( 0, \frac12 \big] \cap \mathbb Q$ the map $L_{c,p}$ admits a Markov partition and in some special cases (including the values $c = \frac1{2^n}$) we can compute the density of $\mu_{p,c}$. In some of these cases we can get a result similar to Theorem~\ref{t:main2}(ii) and compute the frequency of the digits $d$ in the generalised L\"uroth expansions.\\

\vskip .2cm
The paper is organised as follows. In Section \ref{s:skewproduct} we describe how to obtain generalised L\"uroth expansions from the random maps $L_{c,p}$ and we characterise numbers with ultimately periodic expansions. Here we prove Theorem~\ref{t:main1}. Section~\ref{s:analysisapprox} is dedicated to the case $c=0$. Theorem~\ref{t:main2}(i) is proved in Proposition~\ref{p:universal0}, part (ii) is covered by Proposition~\ref{p:le} and part (iii) is done in Proposition~\ref{p:ethetan}. In Section~\ref{s:finitea} we focus on $c>0$. Theorem~\ref{t:main3} corresponds to the content of Theorem~\ref{t:uncountable25} and Theorem~\ref{t:universal}. Proposition~\ref{p:markov} contains the result that $L_{c,p}$ admits a Markov partition for any $c \in \big(0, \frac12 \big] \cap \mathbb Q$ and is followed by various examples in which we explicitly compute the density of the measure $\mu_{c,p}$ and in some cases also the typical speed of convergence of the convergents $\frac{p_n}{q_n}$ as well as the frequency of the digits.

\section{Random $c$-L\"uroth transformations}\label{s:skewproduct}
In this section we introduce the family $\{ L_{c,p} \}_{c \in [0, \frac12], 0 \le p \le 1}$ of random $(c,p)$-L\"uroth transformations and show how these maps produce generalised L\"uroth expansions for all $x \in [c,1]$. Since the probability $p$ does not play a role in this section, we drop the subscript for now and refer to the map $L_c$ as the random $c$-L\"uroth transformation instead. First recall some notation for sequences. Let $\mathcal A $ be an at most countable set of symbols, called {\em alphabet}. Let $\sigma: \mathcal A^\mathbb N \to \mathcal A^\mathbb N$ denote the left shift on sequences, so for a sequence $ a = (a_i)_{i \ge 1} \in \mathcal A^\mathbb N$ we have $\sigma  (a)  =a'$, where $a'_i = a_{i+1}$ for all $i$; with slight abuse of notation we will always use $\sigma$ to denote the left shift on sequences, regardless of the underlying alphabet. For a finite string $a \in \mathcal A^k$ and $1 \le n \le k$ or an infinite sequence $\mathbf{a} \in \mathcal A^\mathbb N$ and $n \ge 1$ we denote by $a_1^n$ the initial part $a_1^n= a_1 \cdots a_n$. We call a sequence $\mathbf{a} = (a_i)_{i \ge 1}$ {\em ultimately periodic} if there exist an $n \ge 0$ and an $r \ge 1$ such that $a_{n+j}=a_{n+r+j}$ for all $j \ge 1$ and {\em periodic} if $n=0$. Finally, we denote cylinder sets in $\mathcal A^{\mathbb{N}}$ using square brackets, i.e., 
\[[b_1, \dots ,b_k]= \{\mathbf{a} \in \mathcal A^{\mathbb{N}}\, : \, a_j=b_j, \, \text{for all } 1 \leq j \leq k\}.\]

\vskip .2cm
For $c \in [0, \frac12]$ and $n \ge 1$ let $z_n = \frac{1}{n}$,
\begin{equation}\label{q:zn}
z_n^+ := z_n + c z_n z_{n-1}  \quad \text{ and } \quad z_{n}^- := z_{n} - c z_n z_{n+1}.
\end{equation}
Then $z_n \le z_n^+ \le z_{n-1}^- \le z_{n-1}$. Define
\[T_{0,c}(x) = \begin{cases} 
T_A(x), & \text{if } x \in \{0\} \cup \bigcup\limits_{n=2}^{\infty} [z_n, z_n^+), 
\\[8pt]
T_L(x), & \text{if } x \in \{ 1 \} \cup \bigcup\limits_{n=2}^{\infty} [z_n^+, z_{n-1}),
\end{cases}\]
and
\[T_{1,c}(x) = \begin{cases} 
T_A(x), & \mbox{if } x \in \{0\} \cup \bigcup\limits_{n=2}^{\infty} [z_n, z_{n-1}^-], \\[8pt]
T_L(x), & \mbox{if } x \in  \{1\} \cup \bigcup\limits_{n=2}^{\infty} (z_{n-1}^-,z_{n-1}).
\end{cases}\]
As can be seen from Figure \ref{f:attractor}, the interval $[c,1]$ is an attractor for the dynamics of both $T_{0,c}$ and $T_{1,c}$, so we choose $[c,1]$ as their domain. Note that $T_{0,c}$ and $T_{1,c}$ assume the same values on the intervals $[z_n, z_n^+)$ and $(z_n^-,z_n]$ and differ on $[z_n^+, z_{n-1}^-]$ for $n > 1$. We denote the union of the subintervals on which $T_{0,c}$ and $T_{1,c}$ assume different values by $S$, i.e., 
\begin{equation}\label{d:switch}
S= [c,1] \cap \bigcup_{n >1}  [z_n^+, z_{n-1}^-],
\end{equation}
and call this the \emph{switch region}. For $c=0$ we see that $S = (0,1]\setminus \{ z_n \, : \, n \ge 1 \}$ and $T_{0,0}=T_L$ and $T_{1,0}=T_A$ except for the points $z_n$ where $T_{0,0}(z_n) = T_{1,0}(z_n)=1$. We combine these two maps to make a random dynamical system by defining the {\em random $c$-L\"uroth transformation}\index{random affine map!$c$-L\"uroth map} $L_c: \{ 0,1\}^\mathbb N \times [c,1] \to  \{ 0,1\}^\mathbb N \times [c,1]$ by
\[L_c(\omega, x) = (\sigma(\omega), T_{\omega_1, c} (x)).\]
See Figure~\ref{f:attractor}(c), for an example with $c=\frac14$.

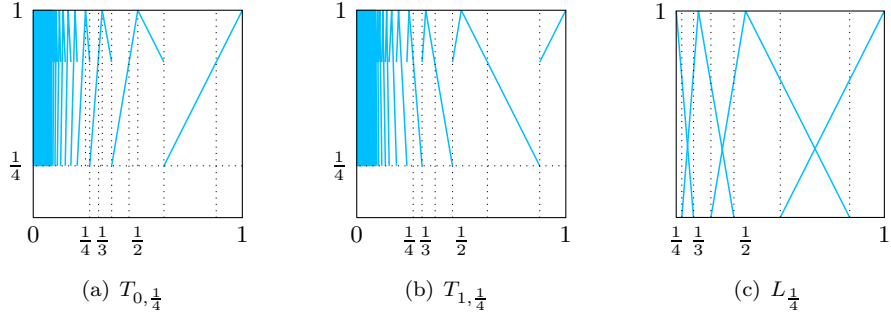
\begin{figure}[h]
\centering
\setcounter{subfigure}{0}
\subfigure[$T_{0, \frac14}$]{
\begin{tikzpicture}[scale=2.75]
\draw[fill=capri, capri]  (0,.25) -- (.09,.25) -- (.09,1) -- (0,1) -- cycle;
\draw[dotted](.25,.25)--(.25,1)(.33,.25)--(.33,1)(.5,.25)--(.5,1);
\draw[line width=0.2mm, capri](.33,1)--(.375,.75)(.5,1)--(.625,.75);
\draw[line width=0.2mm, capri](.09,1)--(.094,.75)(.094,.25)--(.1,1)(.1,1)--(.104,.75)(.104,.25)--(.11,1)(.11,1)--(.115,.75)(.115,.25)--(.125,1)(.125,1)--(.131,.75)(.131,.25)--(.143,1)(.143,1)--(.153,.75)(.153,.25)--(.167,1)(.167,1)--(.18,.75)(.18,.25)--(.2,1)(.2,1)--(.21,.75)(.21,.25)--(.25,1)(.25,1)--(.27,.75)(.27,.25)--(.33,1)(.375,.25)--(.5,1)(.625,.25)--(1,1);
\draw[dotted](.625,0)--(.625,1)(.375,0)--(.375,1)(.27,0)--(.27,1);
\draw[dotted](.875,0)--(.875,1)(.458,0)--(.458,1)(.3125,0)--(.3125,1);
\draw[dotted] (1,.25)--(0,.25)node[left]{\small $\frac14$};
\draw(0,0)node[below]{\small $0$}--(.25,0)node[below]{\small $\frac14$}--(.33,0)node[below]{\small $\frac13$}--(.5,0)node[below]{\small $\frac12$}--(1,0)node[below]{\small $1$}--(1,1)--(0,1)node[left]{\small $1$}--(0,0);
\end{tikzpicture}}
\hspace{5mm}
\subfigure[$T_{1,\frac14}$]{
\begin{tikzpicture}[scale=2.75]
\draw[fill=capri, capri]  (0,.25) -- (.09,.25) -- (.09,1) -- (0,1) -- cycle;
\draw[line width=0.2mm, capri](.25,1)--(.3125,.25)(.33,1)--(.458,.25)(.5,1)--(.875,.25);
\draw[line width=0.2mm, capri](.09,1)--(.096,.25)(.096,.75)--(.1,1)(.1,1)--(.105,.25)(.105,.75)--(.11,1)(.11,1)--(.12,.25)(.12,.75)--(.125,1)(.125,1)--(.137,.25)(.137,.75)--(.143,1)(.143,1)--(.155,.25)(.155,.75)--(.167,1)(.167,1)--(.19,.25)(.19,.75)--(.2,1)(.2,1)--(.24,.25)(.24,.75)--(.25,1)(.25,.25)(.3125,.75)--(.33,1)(.458,.75)--(.5,1)(.875,.75)--(1,1);
\draw[dotted](.625,0)--(.625,1)(.375,0)--(.375,1)(.27,0)--(.27,1);
\draw[dotted](.875,0)--(.875,1)(.458,0)--(.458,1)(.3125,0)--(.3125,1);
\draw[dotted] (1,.25)--(0,.25)node[left]{\small $\frac14$};
\draw(0,0)node[below]{\small $0$}--(.25,0)node[below]{\small $\frac14$}--(.33,0)node[below]{\small $\frac13$}--(.5,0)node[below]{\small $\frac12$}--(1,0)node[below]{\small $1$}--(1,1)--(0,1)node[left]{\small $1$}--(0,0);
\end{tikzpicture}}
\hspace{5mm}
\subfigure[$L_{\frac14}$]{
\begin{tikzpicture}[scale=3.65]
\draw[line width=0.2mm, capri](.25,1)--(.3125,.25)(.33,1)--(.458,.25)(.5,1)--(.875,.25);
\draw[line width=0.2mm, capri](.25,.25)(.27,.25)--(.33,1)(.375,.25)--(.5,1)(.625,.25)--(1,1);
\draw[dotted](.625,.25)--(.625,1)(.375,.25)--(.375,1)(.27,.25)--(.27,1);
\draw[dotted](.875,.25)--(.875,1)(.458,.25)--(.458,1)(.3125,.25)--(.3125,1);
\draw(.25,.25)node[below]{\small $\frac14$}--(.33,.25)node[below]{\small $\frac13$}--(.5,.25)node[below]{\small $\frac12$}--(1,.25)node[below]{\small $1$}--(1,1)--(.25,1)node[left]{\small $1$}--(.25,.25);
\end{tikzpicture}}
\caption{The maps $T_{0,\frac14}$, $T_{1,\frac14}$, and the random $c$-L\"uroth map $L_{\frac14}$ on $\big[\frac14,1\big]$.}
\label{f:attractor}
\end{figure}

\vskip .2cm
To denote the orbit along a specific path $\omega \in \{0,1\}^\mathbb N$ we use the following notation. For a finite string $\omega \in \{0,1\}^k$ and $0 \leq n \leq k$, let 
\begin{equation}\label{q:iterations}
T_{\omega,c}= T_{\omega_k,c} \circ T_{\omega_{k-1},c} \circ \dots \circ T_{\omega_1,c} \quad  \text{and} \quad T_{\omega,c}^n = T_{\omega_1^n,c}= T_{\omega_n,c} \circ T_{\omega_{n-1},c} \circ \dots \circ T_{\omega_1,c}.
\end{equation}
Similarly, for an infinite path $\omega \in \{0,1\}^{\mathbb N}$ and $n \ge 1$ we let
\begin{equation}\label{q:iterationsseq}
 T_{\omega,c}^n = T_{\omega_1^n,c}= T_{\omega_n,c} \circ T_{\omega_{n-1},c} \circ \dots \circ T_{\omega_1,c}.
\end{equation}
For any $\omega$ we also set $T_{\omega,c}^0= T_{ \omega_1^0,c}= id$. Note that we have defined $T_{0,c}$ and $T_{1,c}$ in such a way that $T_{j,c}(x) \neq 0$ for any $j,c,x$. With this notation we can obtain {\em $c$-L\"uroth expansions}\index{$c$-L\"uroth expansion} of points in $[c,1]$ by defining for each $(\omega,x) \in \{0,1\}^\mathbb N \times [c,1]$ two sequences $(s_i)_{i \ge 1}$ (the signs) and $(d_i)_{i \ge 1}$ (the digits) as follows. For $(\omega, x) \in \{0,1\}^{\mathbb{N}} \times [c,1]$ set for $c >0$ and for $i \geq 1$,
\[s_i = s_i (\omega, x) = \begin{cases} 
0, & \mbox{if } L_c^{i-1}(\omega,x) \in  [0] \times S  \cup  \{0,1\}^{\mathbb{N}} \times \big( \cup_n (z_n^-,z_n) \cup \{1\} \big) ,\\[5pt]
1, & \mbox{if } L_c^{i-1}(\omega,x) \in [1] \times S \cup \{0,1\}^{\mathbb{N}} \times \big( \cup_n [z_n,z_n^+) \cup \{0\} \big).
\end{cases}\]
For $c=0$, set $s_i= \omega_i$ for each $i$. For $c \geq 0$, $x \neq 0$ and for $i\ge 1$ set 
\[d_i = d_i(\omega, x)=\begin{cases}
2, & \text{ if } T_{\omega,c}^{i-1}(x)=1,\\
n,& \text{ if } \, T_{\omega,c}^{i-1}(x) \in [z_n, z_{n-1}), \, n \ge 2.
\end{cases}\]
Then one can write for $x \neq 0$ and each $i \ge 1$ that
\[  T_{\omega,c}^i(x)=(-1)^{s_i} d_i(d_i-1) T_{\omega,c}^{i-1}(x) + (-1)^{s_i+1}(d_i-1+s_i).\]
By inversion and iteration we obtain what we call the {\em $c$-L\"uroth expansion} of $(\omega,x)$:
\begin{equation}\label{d:clurothe}
x=\sum_{n=1}^{\infty} (-1)^{\sum_{i=1}^{n-1} s_i(\omega,x)} \frac{d_n(\omega,x)-1 + s_n(\omega,x)}{\prod_{i=1}^{n} d_i(\omega,x)(d_i(\omega,x)-1)},
\end{equation}
where $\sum_{i=1}^0 s_i(\omega,x)=0$ and the sum converges since $d_i\ge 2$ and $T^i_{\omega,c} (x) \in (0,1]$ for all $\omega,x,i$. Note that this expansion of $x$ is of the form \eqref{d:glse}, i.e., it is a generalised L\"uroth expansion.

\begin{remark}\label{r:notintersection}
(i) Due to the fact that $T_{j,c}(x) \neq 0$ for all $j,c$ and $x\neq 0$, the maps $L_c$ do not produce finite generalised L\"uroth expansions. That is, $L_c$ assigns to each $(\omega, x)$ an infinite sequence $(s_n(\omega,x),d_n(\omega,x))_n$ with $s_n \in \{0,1\}$ and $d_n \geq 2$. 

\vskip .2cm
(ii) As an immediate consequence of the above, we see that for each $D \ge 2$ \emph{every} $x \in \big[\frac1D, 1\big]$ has a generalised L\"uroth expansion that only uses digits $d_n \le D$ and that is generated by any random $c$-L\"uroth system with $c \ge \frac1D$. This is in contrast to the deterministic case, where by the result of \v{S}al\'at in \cite{Sal} for any $D$ the set of points $x$ that has $D$ as an upper bound for the L\"uroth digits has Hausdorff dimension strictly less than 1.

\vskip .2cm
(iii) If $x \in S$ and $x \not \in \big\{ \frac{2n-1}{2n(n-1)} \, : \, n \ge 1 \big\}$ (so $T_L(x) \neq T_A(x)$), then $d_2(\omega,x)=2$ for all $\omega$ with $T_{\omega_1}(x) > T_{1-\omega_1}(x)$ and $d_2(\omega,x) >2$ otherwise. For $c=0$ this implies, since $(s_n)_{n \geq 1} = (\omega_n)_{n \geq 1}$ and $S=(0,1] \setminus \{ \frac1n \, : \, n \ge 1\}$, that Lebesgue almost all $x \in [0,1]$ have uncountably many different random L\"uroth expansions. In Section~\ref{s:finitea} we see that a similar statement holds for $c>0$ and we get that most points even have uncountably many different generalised L\"uroth expansions with all $d_n \le D$.
\end{remark}

\vskip .2cm 
Similar to the deterministic case, we call a $c$-L\"uroth expansion of $(\omega,x)$ {\em (ultimately) periodic} if the corresponding sequence $(s_i,d_i)_{i \ge 1}$ is (ultimately) periodic. From the expression \eqref{d:clurothe} it is clear that the $c$-L\"uroth expansion of $(\omega,x)$ is ultimately periodic if and only if there are $n \ge 0$ and $r\ge 1$ such that 
\[T_{\omega,c}^{n+j}(x) = T_{\omega,c}^{n+r+j}(x),\] 
for all $j \ge 0$, implying that $x \in \mathbb Q$. We define the following weaker notion.

\begin{definition}[Returning points\index{returning point}]
Let $c \in \big[0, \frac12 \big]$. We call a number $x \in [c,1]$ {\em returning} for $L_c$ if for every $\omega \in \{0,1\}^{\mathbb{N}}$ there exist an $n=n(\omega) \ge 0$ and an $r = r(\omega) \ge 1$ such that $T_{\omega,c}^n(x)= T_{\omega,c}^{n+r} (x)$. 
\end{definition}

\begin{lemma}\label{p:returningdense}
Let $c  \in \big[ 0, \frac12 \big]$. If $x \in \mathbb{Q} \cap [c,1]$, then $x$ is a returning point for $L_c$. Hence, the set of returning points is dense in $[c,1]$. 
\end{lemma}

\begin{proof}
Let $\frac{P}{Q}$ be the reduced rational representing $x$, for $P, Q \in \mathbb{N}$. Then for any $\omega \in \{0,1\}^{\mathbb{N}}$ and any $t \in \mathbb{N}$ there exist $a,b \in \mathbb{Z}$ such that  $T_{\omega,c}^t (x)= a \frac{P}{Q} + b \in [c,1]$. This implies that $aP+bQ \in \{0,1, \ldots, Q\}$. It then follows by Dirichlet's Box Principle that for some $n \in \mathbb{N}$ there exists an $r \ge 1$ such that $T_{\omega,c}^n(x)=T_{\omega,c}^{n+r} (x)$. 
  \end{proof} 

Contrarily to the deterministic case, the fact that there are $n,r$ with $T_{\omega,c}^n(x)=T_{\omega,c}^{n+r} (x)$ does not necessarily imply that $T_{\omega,c}^{n+j} (x)=T_{\omega,c}^{n+r+j} (x)$ for all $j \geq 0$, since one can make a different choice when arriving at a point in $S$ for the second time. To characterise the ultimately periodic $c$-L\"uroth expansions we define {\em loops}.

\begin{definition}[Loop\index{loop}]
A string $\mathbf{u}$ of symbols in $\{ 0,1\}$ is called a \emph{loop} for $x \in [c,1]$ at the point $y \in [c,1]$ if there exist $\omega \in \{0,1\}^{\mathbb{N}}$ and $n=n(\omega)\ge 0$, $r=r(\omega) \ge 1$ such that
\begin{equation}\label{e:loopdef}
\omega_{n+1}^{n+r} = \mathbf{u}, \quad T_{\omega,c}^n(x)=y= T_{\omega,c}^{n+r}(x) \quad \text{and} \quad T_{\omega,c}^{n+j}(x)\neq y \quad \text{ for } 1 \le j <r.
\end{equation}
We say that $x$ admits the loop $\mathbf u$ at $y$.
\end{definition}

For each $x,y \in [c,1]$ we define an equivalence relation on the collection of loops $\{ \mathbf u \}$ of $x$ at $y$ by setting $\mathbf{u}_1 \sim \mathbf{u}_2$ if the corresponding paths $\omega_1, \omega_2 \in \{0,1\}^{\mathbb{N}}$ satisfying \eqref{e:loopdef} both assign the same strings of signs and digits, i.e., if 
\begin{equation}\label{q:sameLdigits}
(s(\omega_1,x), d(\omega_1,x))_{n(\omega_1)+1}^{n(\omega_1)+r} = (s(\omega_2,x), d (\omega_2,x) )_{n(\omega_2)+1}^{n(\omega_2)+r},
\end{equation}
where $r = r(\omega_1)=r(\omega_2)$. We need this definition since for $x \in  [c,1] \setminus S$, $T_{\omega,c} (x)$ is independent of the choice of $\omega \in \{0,1\}$, i.e., $T_{0,c}(x)= T_{1,c}(x)$ and the corresponding sign and digit only depend on the position of $x$, and not on $\omega$. As a consequence, it is necessary that $T_{\omega,c}^n(x)\in S$ for some $\omega \in \{0,1\}^{\mathbb{N}}$ and $n \geq 0$, to have more than one loop (that is, more than one equivalence class).

\begin{proposition}\label{p:loop}
Let $c \in \big[ 0, \frac12 \big]$ and $x \in \mathbb{Q} \cap [c,1]$. 
\begin{itemize}
\item[(i)] Suppose that there exists an $\omega \in \{0,1\}^{\mathbb N}$ such that $T_{\omega,c}^n(x) \notin S$ for all $n \geq 0$. Then $x$ has a unique and ultimately periodic $c$-L\"uroth expansion. 
\item[(ii)] Suppose that for each $y \in [c,1]$ the point $x$ admits at most one loop at $y$. Then all $c$-L\"uroth expansions of $x$ are ultimately periodic, so there are at most countably many of them.
\item[(iii)] Suppose there is a $y \in [c,1]$, such that $x$ admits at least two loops $\mathbf u_1 \nsim \mathbf u_2$ at $y$. Then $x$ has uncountably many $c$-L\"uroth expansions that are not ultimately periodic, and countably many ultimately periodic $c$-L\"uroth expansions.
\end{itemize}
\end{proposition}

\begin{proof}
For (i) note that if there exists an $\omega \in \{0,1\}^{\mathbb{N}}$ such that $T^n_{\omega,c} (x) \notin S$ for all $n \geq 0$, then $T^n_{\omega,c} (x)$ is independent of the choice of $\omega$ for any $n \geq 0$ and any path $\omega \in \{0,1\}^{\mathbb{N}}$ yields the same $c$-L\"uroth expansion. The result then follows from Lemma~\ref{p:returningdense}.

\vskip .2cm
For (ii) suppose by contradiction that there exists an $\omega \in \{0,1\}^{\mathbb{N}}$ such that $(\omega,x)$ presents a $c$-L\"uroth expansion that is not ultimately periodic. Since $x \in \mathbb Q$, the set $\{T_{\omega,c}^n(x)\}_{n \geq 0}$ consists of finitely many points, and so, in particular, there exists a point $y = T_{\omega,c}^j(x)$ for some $j \geq 0$, that is visited infinitely often. Let $\{j_i\}_{i \in \mathbb{N}}$ be the sequence such that $T_{\omega,c}^{j_i} = y$ for every $j_i$. Since $(\omega, x)$ does not have an ultimately periodic expansion, there exists a $k$ such that 
\[( s(\omega, x), d(\omega, x))_{j_{k-1}+1}^{j_k} \neq ( s(\omega, x), d(\omega, x))_{j_k+1}^{j_{k+1}},\]
which means in particular that the loops $\omega_{j_{k-1}+1}^{j_k}$ and $\omega_{j_{k}+1}^{j_{k+1}} $ are not in the same equivalence class, contradicting the assumption on the number of admissible loops at $y$. The second part follows since the sequence $(s_n(\omega,x),d_n(\omega,x))_n$ takes its digits in an at most countable alphabet. %The second part follows since there exist only countably many ultimately periodic sequences over any given countable alphabet.

\vskip .2cm
For (iii) let $\mathbf u_1$ and $\mathbf u_2$ be two loops of $x$ at $y$ with $\mathbf u_1 \nsim \mathbf u_2$. Consider $\omega \in \{0,1\}^{\mathbb{N}}$ satisfying \eqref{e:loopdef}, i.e., such that 
\[\omega_{n+1}^{n+r} = \mathbf{u}_1, \quad \text{and} \quad T_{\omega,c}^n(x)=T_{\omega,c}^{n+r}(x)= y,\]
for some $n,r \in \mathbb{N}$. Now take any sequence $(\ell_j)_{j \ge 1} \subseteq \mathbb N^{\mathbb N}$ and consider the path $\tilde{\omega} \in \{0,1\}^{\mathbb{N}}$ defined by the concatenation
\[\tilde{\omega}= \omega_1^n \mathbf{u}_1^{\ell_1} \mathbf u_2^{\ell_2} \mathbf{u}_1^{\ell_3} \mathbf{u}_2^{\ell_4} \mathbf u_1^{\ell_5} \mathbf u_2^{\ell_6} \ldots  \ .\]
If $(\ell_j)$ is not ultimately periodic, then it is guaranteed by \eqref{q:sameLdigits} that the sequence $(s_i(\tilde{\omega},x), d_i(\tilde{\omega},x))_{i \in \mathbb{N}}$ is not ultimately periodic and as a result $(\tilde{\omega},x)$ presents a $c$-L\"uroth expansion that is not ultimately periodic. Since there are uncountably many such sequences $(\ell_j)$ each yielding a different corresponding sequence of signs and digits, the first part of the statement follows. Taking any ultimately periodic sequence $(\ell_j)$ instead will yield an ultimately periodic $c$-L\"uroth expansion.
\end{proof}

\begin{proposition}\label{p:iinfinite}
Let $c \in \big[ 0, \frac12 \big]$. If $x \in [c,1] \setminus \mathbb{Q}$, then any $c$-L\"uroth expansion of $x$ is infinite and not ultimately periodic. 
\end{proposition}

\begin{proof}
Let $x \in [c,1] \setminus \mathbb{Q}$ be given and assume that there exists an $\omega \in \{0,1\}^{\mathbb{N}}$ such that the corresponding $c$-L\"uroth expansion of $(\omega,x)$ is ultimately periodic. Then there is an $n \ge 0$ and an $r \ge 1$ such that $T_{\omega,c}^n(x)=T_{\omega,c}^{n+r}(x)$. Hence, there are $a, b, c, d \in \mathbb{Z}$ such that $ax+b= cx+d$, implying that $x \in \mathbb{Q}$, contradicting the choice of $x$.
\end{proof}

\begin{example}\label{e:rationalexp}
To illustrate Proposition~\ref{p:loop} we give an example of the various possibilities for periodicity of expansions of rational numbers. Let $c=\frac13$ and first consider the rational number $\frac67$. See Figure \ref{f:L13}(a) for the random map $L_{\frac13}$  with the possible orbits of $\frac67$. Figure \ref{f:L13}(b) is a visualisation of the random orbits of $\frac67$. We explicitly identify paths $\omega \in \{0,1\}^{\mathbb{N}}$ that produce $c$-L\"uroth expansions of $\frac67$ that are periodic, ultimately periodic and not ultimately periodic and list them in Table \ref{t:rationalomega}. 

\begin{figure}[h!]
\centering
\subfigure[]{
\begin{tikzpicture}[scale=5]
\draw(.33,.33)node[below]{\small $\frac13$}--(.5,.33)node[below]{\small $\frac12$}--(1,.33)node[below]{\small $1$}--(1,1)node[right]{\small $1$}--(.33,1)--(.33,.33);
\draw[line width=0.2mm, black](.33,1)--(.44,.33)(.5,1)--(.83,.33);
\draw[line width=0.2mm, black](.38,.33)--(.5,1)(.66,.33)--(1,1);
\draw[dotted](.38,.33)--(.38,1)(.44,.33)--(.44,1)(.66,.33)--(.66,1)(.83,.33)--(.83,1)(.5,.33)--(.5,1)(.33,.66)--(1,.66);
\draw[dashed, line width=0.2mm, red](.857,.33)--(.857,.857)(.857,.714)--(.714,.714)(.714,.714)--(.714,.33)(.714,.33)--(.714,.428)(.714,.428)--(.428,.428)(.714,.571)--(.428,.571)(.571,.571)--(.571,.33)(.428,.571)--(.428,.33)(.571,.571)--(.571,.857)(.571,.857)--(.857,.857);
\draw[dashed, line width=0.2mm, black](.33,.33)--(1,1);
\filldraw[red] (.857,.714) circle (.2pt);
\filldraw[red] (.714,.428) circle (.2pt);
\filldraw[red] (.714,.571) circle (.2pt);
\filldraw[red] (.428,.428) circle (.2pt);
\filldraw[red] (.428,.571) circle (.2pt);
\filldraw[red] (.571,.857) circle (.2pt);
\draw (.857,.33) node[below] {\small $\frac67$};
\draw (.714,.33) node[below] {\small $\frac57$};
\draw (.571,.33) node[below] {\small $\frac47$};
\draw (.428,.33) node[below] {\small $\frac37$};
\end{tikzpicture}}
\hspace{1cm}
\subfigure[]{
\resizebox{.25\textwidth}{!}{
\begin{tikzpicture}[->,>=stealth',shorten >=1pt,auto,node distance=2cm,semithick]
%->,>=stealth',shorten >=1pt,auto,node distance=2.4cm, semithick
\tikzstyle{capri}=[rectangle,fill=none,draw=capri,text=black]
\tikzstyle{babyblueeyes}=[rectangle,fill=none,draw=babyblueeyes,text=black]
\tikzstyle{brightpink}=[rectangle,fill=none,draw=brightpink,text=black]
\tikzstyle{green}=[rectangle,fill=none,draw=green,text=black]

\node[capri](A){\footnotesize $\frac67$};
\node[](B)[right=0.7cm of A]{\footnotesize $\frac57$};
\node[green](C)[right=0.7cm of B]{\footnotesize $\frac37$};
%\node[brightpink](D)[right=0.7cm of C]{\footnotesize $\frac47$};  
\node[brightpink](E)[below=0.7cm of B]{\footnotesize $\frac47$}; 
\node[white](F)[below=0.25cm of E]{}; 
%\node[green](G)[below=0.7cm of C]{\footnotesize $\frac37$}; 
%\node[capri](H)[right=0.7cm of D]{\footnotesize $\frac67$};

\draw[->] (C) to[loop above] node[above]{\scriptsize $1$}(C);

\path (A) edge node{}(B)
        (B) edge node{\scriptsize $0$}(C)
 %       (C) edge node{\footnotesize $0$}(D)
        (B) edge node{\scriptsize $1$}(E)
        (E) edge node{}(A)
 %       (D) edge node{}(H)
        (C) edge node{\scriptsize $0$}(E);
        \end{tikzpicture}}}

\caption{The random L\"uroth map $L_c$ for $c=\frac13$ with the random orbits of $\frac67$ in red in (a) and another visualisation of the orbits of $\frac67$ in (b). The digits with the arrows indicate which one of the maps $T_{0,c}$ or $T_{1,c}$ is applied. If there is no digit, then both maps yield the same orbit point.}
\label{f:L13}
\end{figure}
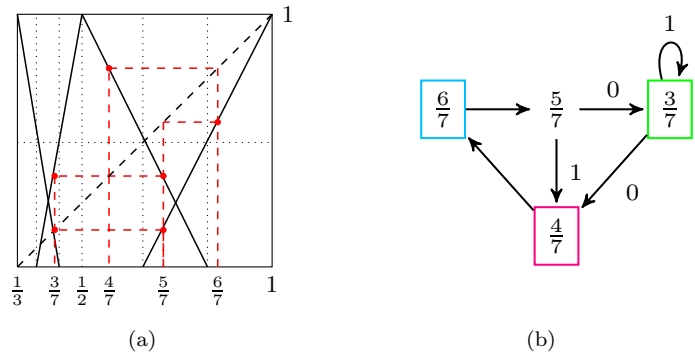

\begin{table}[h!]
\centering
\renewcommand{\arraystretch}{2}
\begin{tabular}{l | l}
$\boldsymbol{\omega}$ & \textbf{Expansion of } $\frac67$\\
\hline
 $(011)^{\infty}$  & $((0,2),(1,2)^2)^\infty$ is periodic\\
 $001^{\infty}$  & $((0,2)^2,(1,3)^\infty)$ is ultimately periodic\\
 $0^210^41^20^41^30^4 1^4 \ldots$ & $((0,2)^2,(1,3),(0,3),(1,2),(0,2)^2,(1,3)^2,(0,3),(1,2),(0,2)^2, (1,3)^3, \ldots)$\\
 & is not ultimately periodic
 \end{tabular}
 \caption{Examples of $\omega$'s and the corresponding type of the $c$-L\"uroth expansions.}
 \label{t:rationalomega}
\end{table}

For the point $\frac34$ it holds that for any $\omega \in \{0,1\}^\mathbb N$, $T_{\omega,c}^1\big(\frac34\big) = \frac12$ and $T_{\omega,c}^n\big(\frac34\big) = 1$ for any $n \ge 2$. Hence, $\frac34$ has precisely two $c$-L\"uroth expansions (for any $c$) that are given by the sequences
\[ ((0,2),(1,2),(0,2)^\infty) \quad \text{ and } \quad ((1,2)^2,(0,2)^\infty).\]
Hence in case (ii) of Proposition~\ref{p:loop} there are rational $x$ that have only a finite number of ultimately periodic expansions and we cannot improve on the statement without giving a further description of the specific positions of the random orbit points.
\end{example}

\begin{remark}
Lemma \ref{p:returningdense} gives a further bound on the number of admissible digits $(d_n)_{n \in \mathbb{N}}$ in the $c$-L\"uroth expansions of a rational number $x$. More precisely, if $c=0$, any $c$-L\"uroth expansion of $x=\frac{P}{Q} \in \mathbb{Q} \cap [0,1]$ has at most $Q+1$ different digits. Differently, by Proposition \ref{p:iinfinite}, for irrational numbers the set of admissible digits is $\mathbb{N}_{\geq 2}$. Note that for $ \frac{1}{\ell +1} \leq c < \frac{1}{\ell}$, the bound is given by the minimum between the previous quantities and $\ell$.
\end{remark}

The proof of Theorem~\ref{t:main1} is now given by Proposition~\ref{p:loop} and Proposition~\ref{p:iinfinite}.

\section{Approximations of irrationals}\label{s:analysisapprox}
The approximation properties of $c$-L\"uroth expansions can be studied via the dynamical properties of the associated random system. For this one needs to have an accurate description of an invariant measure for the random system. Let $0 <p<1$. The vector $(p,1-p)$ represents the probabilities with which we apply the maps $T_{0,c}$ and $T_{1,c}$ respectively. One easily check that the probability measure $m_p \times \lambda$, where $m_p$ is the $(p,1-p)$-Bernoulli measure on $\{0,1\}^{\mathbb{N}}$ and $\lambda$ is the one-dimensional Lebesgue measure on $[0,1]$ is invariant and ergodic for $L_0$. Therefore, in this section we focus on $c=0$, fix a $p$ and drop the subscripts $c,p$, so we write $L = L_{0,p}$. We first prove a result on the number of different L\"uroth expansions that $L$ produces for Lebesgue almost all $x \in [0,1]$ and then investigate two ways of quantifying the approximation properties of all these expansions.

\subsection{Universal generalised L\"uroth expansions}
The random dynamical system $L=L_{0,p}$ is capable of producing for each number $x \in [0,1]$ essentially all expansions generated by all the members of the family of GLS transformations studied in \cite{BaBu}, i.e., GLS transformations with \emph{standard L\"uroth partition}, given by $\mathcal P_L = \big\{ \big[ \frac1{n}, \frac{1}{n-1} \big) \big\}_{n \ge 2}$. In the previous section we mentioned that Lebesgue almost every $x$ has uncountably many different $0$-L\"uroth expansions and thus uncountably many different generalised L\"uroth expansions. Here we prove an even stronger statement. 

\vskip .2cm
Let $\mathcal A = \{ (s,d) \, : \, s\in \{0,1\}, \, d \in \mathbb N_{\ge 2} \}$ be the alphabet of possible digits for generalised L\"uroth expansions. Define the map $\psi: \mathcal A^\mathbb N \to (0,1]$ by
\begin{equation}\label{q:psi}
\psi \big( ((s_n,d_n))_{n \ge 1} \big) =  \sum_{n=1}^{\infty} (-1)^{\sum_{i=1}^{n-1} s_i} \frac{d_n-1 + s_n}{\prod_{i=1}^{n} d_i(d_i-1)}
\end{equation}
To show that this map is well-defined and surjective, set for any $(s,d) \in \mathcal A$
\[ \bar \Delta (s,d) = [s] \times \Big[ \frac1d, \frac{1}{d-1} \Big],\]
for the cylinder set $[s]\subseteq  \{0,1\}^\mathbb N$ and the interval $\big[ \frac1d, \frac1{d-1} \big] \subseteq [0,1]$. Then $m_p \times \lambda (\bar \Delta(s,d)) = \frac{(-1)^s(p-s)}{d(d-1)}>0$. Take any sequence $((s_n,d_n))_n \in \mathcal A^\mathbb N$ that does not end in $((0,d+1),(0,2)^\infty)$ for some $d \geq 2$. The set
\begin{equation}\label{q:limitcylinder}
\lim_{n \to \infty}\bar \Delta(s_1,d_1) \cap L^{-1}\bar \Delta(s_2,d_2) \cap \cdots \cap L^{-(n-1)} \bar \Delta(s_n,d_n) \subseteq \{0,1\}^\mathbb N \times [0,1]
\end{equation}
is non-empty as a countable intersection of closed sets. Moreover,
\[ \lim_{n \to \infty} m_p \times \lambda \big(\bar  \Delta(s_1,d_1) \cap L^{-1}\bar \Delta(s_2,d_2) \cap \cdots \cap  L^{-(n-1)} \bar \Delta(s_n,d_n) \big) = \lim_{n \to \infty} \frac{p^{n-\sum_{i=1}^n s_i} (1-p)^{\sum_{i=1}^n s_i}}{\prod_{i=1}^n d_i(d_i-1)}=0,\]
so the set from \eqref{q:limitcylinder} consists of precisely one point, call it $(\omega,x)$. Then $\omega = (s_n)_n \in \{0,1\}^\mathbb N$ and by the assumption that $((s_n,d_n))_n$ does not end in $((0,d+1),(0,2)^\infty)$ it follows that $s_n(\omega,x) = s_n$ and $d_n(\omega,x)=d_n$ for all $n \ge 1$, so that
\[ x = \sum_{n=1}^{\infty} (-1)^{\sum_{i=1}^{n-1} s_i} \frac{d_n-1 + s_n}{\prod_{i=1}^{n} d_i(d_i-1)} \in [0,1].\]
One easily checks that for any $d \ge 2$,
\[ \psi \big( ((0,d+1),(0,2)^\infty) \big)= \psi \big( ((1,d),(0,2)^\infty) \big).\]
%Note that if there is a $k \ge 1$, such that $(s_k,d_k) = (0,d+1)$ and  $(s_n,d_n)=(0,2)$ for all $n \ge k+1$, then
%\[ \begin{split}
%\sum_{n=1}^{\infty} (-1)^{\sum_{i=1}^{n-1} s_i} \frac{d_n-1 + s_n}{\prod_{i=1}^{n} d_i(d_i-1)} =\, & \sum_{n=1}^{k-1}(-1)^{\sum_{i=1}^{n-1} s_i} \frac{d_n-1 + s_n}{\prod_{i=1}^{n} d_i(d_i-1)}\\
%& + \frac{(-1)^{\sum_{i=1}^{k-1}s_i}}{\prod_{i=1}^{k-1}d_i(d_i-1)} \Big( \frac{d}{d+1} + \frac1{(d+1)d} \sum_{j \ge 1} \frac1{\prod_{i=1}^j 2 \cdot 1} \Big)\\
%=\, & \sum_{n=1}^{k-1}(-1)^{\sum_{i=1}^{n-1} s_i} \frac{d_n-1 + s_n}{\prod_{i=1}^{n} d_i(d_i-1)} + \frac{(-1)^{\sum_{i=1}^{k-1}s_i}}{d\prod_{i=1}^{k-1}d_i(d_i-1)}\\
%=\, & \sum_{n=1}^{k-1}(-1)^{\sum_{i=1}^{n-1} s_i} \frac{d_n-1 + s_n}{\prod_{i=1}^{n} d_i(d_i-1)}\\
%& + \frac{(-1)^{\sum_{i=1}^{k-1}s_i}}{\prod_{i=1}^{k-1}d_i(d_i-1)} \Big( \frac{d}{d-1} - \frac1{d(d-1)} \sum_{j \ge 1} \frac1{\prod_{i=1}^j 2 \cdot 1} \Big).
%\end{split}\]
%In other words, the sequences $((0,d+1),(0,2)^\infty)$ and $((1,d),(0,2)^\infty)$ correspond to the same number $x$. Therefore, the map $\psi: \mathcal A^\mathbb N \to (0,1]$ given by
%\begin{equation}\label{q:psi}
%\psi \big( ((s_n,d_n))_{n \ge 1} \big) =  \sum_{n=1}^{\infty} (-1)^{\sum_{i=1}^{n-1} s_i} \frac{d_n-1 + s_n}{\prod_{i=1}^{n} d_i(d_i-1)}
%\end{equation}
Therefore, $\psi$ is well defined and surjective and by the way we defined the maps $T_{0,0}$ and $T_{1,0}$ on any of the points $z_n$ we get that any sequence $ ((s_n,d_n))_n \in \mathcal A^\mathbb N$ that does not end in $((0,d+1),(0,2)^\infty)$ corresponds to a 0-L\"uroth expansion of a number $x \in (0,1]$. We call a sequence $((s_n,d_n))_n \in \mathcal A^\mathbb N$ with $\psi \big( ((s_n,d_n))_{n \ge 1} \big) =x$ a {\em universal} 0-L\"uroth expansion of $x$ if every finite block $(t_1,b_1), \ldots, (t_j,b_j) \in \mathcal A^j$ occurs in the sequence $((s_n,d_n))_n$. Note that the sequences ending in $((0,d+1),(0,2)^\infty)$ and $((1,d),(0,2)^\infty)$ can not be universal.

\begin{proposition}\label{p:universal0}
Lebesgue almost all $x \in [0,1]$ have uncountably many universal 0-L\"uroth expansions.
\end{proposition}

\begin{proof}
Define the set
\[ N = \{ x \in (0,1] \, : \, \exists \, \omega \in \{0,1\}^\mathbb N, \, k\ge 0,\, n\ge 2, \, \, T_\omega^k (x) = z_n\}.\]
Then by Theorem~\ref{t:main1}, $N \subseteq \mathbb Q$, so $m_p \times \lambda (\{0,1\}^\mathbb N \times N)=0$. For any $(i,j) \in \mathcal A$ set
\[ \Delta (s,d) = [s] \times \Big( \frac1d, \frac1{d-1} \Big).\]
For any $(\omega,x) \in \{ 0,1\}^\mathbb N \times (0,1] \setminus N$ the block $(t_1,b_1), \ldots, (t_j,b_j)$ occurs in position $k \ge 1$ of the sequence $ ((s_n(\omega,x),d_n(\omega,x)))_n $ precisely if
\[ L^{k-1} (\omega,x) \in \Delta(t_1,b_1) \cap L^{-1} \Delta(t_2,b_2) \cap \cdots \cap L^{-(j-1)} \Delta(t_j,b_j).\]
Since $m_p \times \lambda \big( \Delta(t_1,b_1) \cap L^{-1}\Delta(t_2,b_2) \cap \cdots \cap L^{-(j-1)} \Delta(t_j,b_j) \big) >0$ it follows from the ergodicity of $L$ that $m_p \times \lambda$-a.e.~$(\omega,x)$ eventually enters this set, so the set of points $(\omega,x)$ for which the 0-L\"uroth expansion does not contain $(t_1,b_1), \ldots, (t_j,b_j)$ is a $m_p \times \lambda$-null set. There are only countably many different blocks $(t_1,b_1), \ldots, (t_j,b_j)$, so the set of points $(\omega,x)$ that have a non-universal 0-L\"uroth expansion also has measure 0. From Fubini's Theorem we get the existence of a set $B \subseteq (0,1]\setminus N$ with $\lambda(B)=1$ with the property that for each $x \in B$ a set $A_x \subseteq \{0,1\}^\mathbb N$ exists with $m_p (A_x) =1$ and such that for any $(\omega,x) \in A_x \times \{x\}$ the sequence $(s_n(\omega,x),d_n(\omega,x))_n$ is a universal 0-L\"uroth expansion of $x$. The set $A_x$ has full measure, so contains uncountably many $\omega$'s. Let $x \in B$ and $\omega, \tilde \omega \in A_x$. Then $x \not \in N$ and thus if $\omega_n \neq \tilde \omega_n$ for some $n \ge 1$, then $s_n(\omega,x) \neq s_n(\tilde \omega,x)$. Hence, the sequences $\omega \in A_x$ all give different universal 0-L\"uroth expansions for $x$.
\end{proof}

%Note that if $x$ has a universal 0-L\"uroth expansion, then $x \not \in \big\{ \frac{2n-1}{2n(n-1)} \, : \, n \ge 1 \big\}$. By Remark~\ref{r:notintersection}(iii) the fact that $s_n(\omega,x) \neq s_n(\tilde \omega,x)$ for some $n$ then implies that $d_{n+1}(\omega,x) \neq d_{n+1}(\tilde \omega,x)$.

\subsection{Speed of convergence}

Recall that for each $n \geq 1$ the $n$-th convergent $\frac{p_{n}}{q_{n}} (\omega,x)$ of $(\omega, x)$ is given by
\[\frac{p_n}{q_n}=\frac{p_{n}}{q_{n}} (\omega,x)= \sum_{k = 1}^n (-1)^{\sum_{i=1}^{n-1} \omega _i} \prod_{i=1}^k \frac{d_k-1+\omega_k}{d_i(d_i-1)},\]
so that the study of the quantity $\big| x - \frac{p_{n}}{q_{n}} \big|$ provides information on the quality of the approximations of $x$ obtained from $L$. In the following, we take two different perspectives. In this section we compute the pointwise Lyapunov exponent $\Lambda$ and in the next section we consider the approximation coefficients $\theta_n= \theta_n^{0,p}$ similar to the ones defined in \eqref{q:appcf}.\\

Let $I$ be an interval of the real line, $\Omega \subseteq \mathbb{N}$ and $\pi : \Omega^{\mathbb N}\times I \to I$ the canonical projection onto the second coordinate. The following definition can be found in \cite{GBL, Be}, for example.

\begin{definition}[Lyapunov exponent\index{Lyapunov exponent}]
For any random interval map $R: \Omega^{\mathbb{N}} \times I \to \Omega^{\mathbb{N}} \times I$ and for any point $(\omega, x) \in \Omega^{\mathbb{N}} \times I$, the \emph{pointwise Lyapunov exponent} is defined as
\begin{equation}\label{d:Lyap}
\Lambda (\omega, x) := \lim_{n \rightarrow \infty} \frac1n  \log \abs{ \frac{\de}{\de x} \pi (R^n(\omega, x))} ,
\end{equation}
whenever the limit exists.
\end{definition}

We use this to determine the speed of convergence of the sequence $\big( \frac{p_n}{q_n}(\omega,x) \big)_n$ to a number $x \in [0,1]$ for which $T_\omega^n (x) \neq 0$ for all $n\ge 0$. By Proposition~\ref{p:loop} and Proposition~\ref{p:iinfinite} this includes all irrational $x$ and part of the rationals. For any such $x$ and each sequence $((s_n(\omega,x),d_n(\omega,x))_{n \ge 1} \in \mathcal A^\mathbb N$ of signs and digits the rational $\frac{p_n}{q_n}(\omega,x)$ is one of the endpoints of the projection onto the unit interval of the cylinder $[(s_1,d_1), \ldots, (s_n,d_n)]$, i.e., one of the endpoints of the interval
\[ \psi ( [(s_1,d_1), \ldots, (s_n,d_n)]).\]
Since the map $T_\omega^n$  is surjective and linear on this interval and $\frac{\de}{\de x}\, T_\omega^n (x)= \prod_{k=1}^n d_k(d_k-1)$,
the Lebesgue measure of this interval is $\big( \prod_{k=1}^n d_k(d_k-1) \big)^{-1}$. The following result compares to Theorem~\ref{t:bi} from \cite{BI09}.
\begin{proposition}\label{p:le}
For any $(\omega, x) \in \{0,1\}^{\mathbb{N}} \times [0,1]$ with $T^n_\omega(x) \neq 0$ for all $n \ge 0$ the speed of approximation of the random L\"uroth map $L$ is given by
\[\abs{x- \frac{p_{n}}{q_{n}} (\omega,x)} \asymp \exp (-n \Lambda (\omega, x)). \]
In particular, for $m_p \times \lambda$-a.e.~$(\omega, x)$
\[\Lambda (\omega, x) = \sum_{d=2}^{\infty} \frac{\log (d (d-1))}{d (d-1)} \cong 1.98329 \ldots \ . \]
Furthermore, the map $\Lambda: \{0,1\}^{\mathbb{N}} \times [0,1] \to [ \log 2, \infty ) $ is onto. 
\end{proposition}

\begin{proof}
Let $(\omega,x)$ be such that $T^n_\omega (x) \neq 0$ for each $n$. Write $s_n = s_n(\omega,x)$ and $d_n = d_n(\omega,x)$ for each $n \ge 1$. Since $\frac{p_n}{q_n}(\omega,x)$ is one of the endpoints of the interval $ \psi ( [(s_1,d_1), \ldots, (s_n,d_n)]) $ it follows that
\[ \abs{x- \frac{p_{n}}{q_{n}}(\omega,x)} \leq  \prod_{k=1}^n \frac1{d_k( \omega,x)(d_k( \omega,x)-1)} .\]
We claim that
\begin{equation}\label{q:pqn-1}
\abs{x- \frac{p_{n}}{q_{n}}(\omega,x)} \geq  \prod_{k=1}^{n+1} \frac1{d_k( \omega,x)(d_k( \omega,x)-1)}.
\end{equation}
To see this, note that for $n=1$ we have
\[ \frac1{d_2} \le T_\omega^1 (x) = (-1)^{s_1} d_1(d_1-1) x + (-1)^{s_1+1} (d_1-1+s_1) \le \frac1{d_2-1}.\]
Hence,
\[  \frac1{d_1(d_1-1)d_2} \le (-1)^{s_1} \Big( x - \frac{d_1-1+s_1}{d_1(d_1-1)} \Big) \le \frac1{d_1(d_1-1)(d_2-1)}.\]
Since $\frac{p_1}{q_1} = \frac{d_1-1+s_1}{d_1(d_1-1)}$ it follows that
\[ \abs{x- \frac{p_1}{q_1}} \ge \frac1{d_1(d_1-1)d_2(d_2-1)}.\]
In the same way, from 
\[ \frac1{d_{n+1}} \le T_\omega^n (x) = (-1)^{s_n} d_n(d_n-1) T_\omega^{n-1} (x) + (-1)^{s_n+1} (d_n-1+s_n) \le \frac1{d_{n+1}-1},\]
we obtain
\[ \abs{T_\omega^{n-1} (x)- \frac{d_n-1+s_n}{d_n(d_n-1)}} \ge \frac1{d_n(d_n-1)d_{n+1}(d_{n+1}-1)}.\]
By the definition of the convergents, we then have
\begin{equation}\nonumber
\begin{split} 
\abs{x- \frac{p_{n}}{q_{n}}(\omega,x)}  &= \frac{T_\omega^n(x)}{\prod_{i=1}^{n} d_i (d_i-1) }  \\[10pt]
&= \frac{(-1)^{s_n} d_n(d_n-1) T_\omega^{n-1} (x) + (-1)^{s_n+1} (d_n-1+s_n)}{\prod_{k=1}^{n} d_i (d_i-1)  } \\[10pt]
&= \abs{T_\omega^{n-1} (x)  -  \frac{(d_n-1+s_n)}{d_n(d_n-1)} } \cdot  \prod_{k=1}^{n-1}  \frac{1}{d_i (d_i-1)} \\[10pt]
& \ge \prod_{k=1}^{n+1} \frac1{ d_i(d_i-1)}.
\end{split}
\end{equation}
This gives the claim. 
%On the other hand, Proposition~\ref{p:iinfinite} gives that $|x -\frac{p_n}{q_n}(\omega,x)| >0$ for each $n$. Since
%\[ \lim_{n \to \infty} \prod_{k=1}^n \frac1{d_k( \omega,x)(d_k( \omega,x)-1)} =0,\]
%for each $n$ we can find an $m$ such that
%\[ \abs{x- \frac{p_{n}}{q_{n}}(\omega,x)} \geq  \prod_{k=1}^{n+m} \frac1{d_k( \omega,x)(d_k( \omega,x)-1)}.\]
It follows that
\begin{equation} \nonumber
\begin{split}
\lim_{n \rightarrow \infty} \frac1n \log \abs{x- \frac{p_{n}}{q_{n}}(\omega,x)} &= \lim_{n \rightarrow \infty} \frac1n \log \bigg( \prod_{k=1}^n \frac1{d_k( \omega,x)(d_k( \omega,x)-1)} \bigg)\\
%&= \lim_{n \rightarrow \infty} \frac1n \log \lambda \big( \psi ( [(s_1,d_1), \ldots, (s_n,d_n)])\big)\\
&=\lim_{n \rightarrow \infty} \frac1n \log \prod_{k=1}^{n} \lambda \big( \psi ([s_k,d_k])\big)\\
&=\lim_{n \rightarrow \infty} \frac1n \log \abs{\frac{\de}{\de x} \pi(L^n(\omega,x))}^{-1}\\
&=-\Lambda (\omega, x).
\end{split}
\end{equation}
So $\Lambda$ measures the asymptotic exponential growth of approximation, i.e., 
\[\abs{x- \frac{p_{n}}{q_{n}} (\omega,x)} \asymp \exp (-n \Lambda (x,\omega)).\]
Recall now that $m_p \times \lambda$ is invariant and ergodic for $L$. Applying Birkhoff's Ergodic Theorem we get for $(m_p \times \lambda)$-a.e.~$(\omega, x)$ that
\begin{equation}\label{e:lexp}
\begin{split}
\Lambda (\omega,x) = \lim_{n \rightarrow \infty} \frac1n \log \abs{ \frac{\de}{\de x} T_{\omega}^n (x) } &= \lim_{n \rightarrow \infty} \frac1n \log \abs{ \prod_{k=1}^{n} \frac{\de}{\de x} T_{\omega_k} ( T_{\omega}^{k-1}(x))} \quad\\
%&= \lim_{n \rightarrow \infty} \frac1n \log \abs{ \prod_{k=0}^{n-1} \frac{\de}{\de x} T_{\omega_{k+1}} ( T_{\omega}^{k}(x)) } \\
%&= \lim_{n \rightarrow \infty} \frac1n \log \abs{ \prod_{k=0}^{n-1} \frac{\de}{\de x} \pi \circ L ( \sigma^{k} (\omega),  T_{\omega}^{k}(x))} \\
%&= \lim_{n \rightarrow \infty} \frac1n \log \abs{ \prod_{k=0}^{n-1} \frac{\de}{\de x} \pi L ( L^k (\omega, x))} \\
&=\int_{\{0,1\}^{\mathbb{N}} \times [0,1]} \log \abs{\frac{\de}{\de x} \pi (L(\omega, x))} d (m_p \times \lambda ) (\omega,x) \\
%&=\int_{\{0,1\}^{\mathbb{N}} \times [0,1]} \log \abs{\frac{\de}{\de x} T_{\omega} (x)} d (m_p \times \lambda ) (\omega,x) \\
&= \sum_{d=2}^{\infty} \frac{\log (d (d-1))}{d (d-1)}=: \Lambda_{m_p \times \lambda}.
\end{split}
\end{equation}
This gives the second part of the proposition. For the last part, it is obvious that $\log 2$ is a lower bound for $\Lambda (\omega,x)$, since $|\frac{\de}{\de x} T_j (x)| \ge 2$ for $j=0,1$. The rest follows from Theorem~\ref{t:bi} since the sequence $\omega= (\overline{0})$ reduces $L$ to the standard L\"uroth map $T_L$.
\end{proof}

\begin{remark}
It follows from \eqref{e:lexp} that the set of points (of zero Lebesgue measure) that present a Lyapunov exponent different from $\Lambda_{m_p \times \lambda}$, includes {\em fixed points}, i.e., points $(\omega, x) \in \{0,1\}^{\mathbb{N}} \times [0,1]$ for which $T_\omega^k(x)=x$ for all $k \geq 1$, since for any such point $(\omega,x)$, if $x \in [z_d, z_{d-1})$, then $\Lambda (\omega, x) = \log d$. 
%\[\prod_{k=1}^{n} \frac{\de}{\de x} T_{\omega_{k+1}} ( T_{\omega}^{k}(x)) = \prod_{k=1}^{n} \frac{\de}{\de x} T_{\omega_{k+1}} ( x) = \bigg( \frac{\de}{\de x} T_{\omega} ( x)  \bigg)^n,\]
%for all $n \geq 1$. Hence, if $x \in [z_d, z_{d-1})$ then
%\begin{equation} \nonumber
%\begin{split}
%\Lambda (\omega, x) &= \lim_{n \rightarrow \infty} \frac1n \log \abs{ \prod_{k=1}^{n} \frac{\de}{\de x} \pi \circ L ( L^k(\omega,x))}\\
%&=\lim_{n \rightarrow \infty} \frac1n \log \abs{d^n}\\
%&= \log d.
%\end{split}
%\end{equation}
Fixed points can be found in any set $\{0,1\}^\mathbb N \times [z_d, z_{d-1})$ for $d \geq 2$, which immediately yields that $\Lambda$ attains arbitrarily large values.% starting from $\log 2 \cong 0.69314 \ldots$, which corresponds to the Lyapunov exponent of the fixed point given by the sequence $(\, \overline{(0, 2)} \, )$.
\end{remark}

\subsection{Approximation coefficients}
A GLS map $T_{\mathcal P,\varepsilon}: [0,1] \rightarrow [0,1]$ is defined by a partition $\mathcal P=\{I_n = (\ell_n, r_n]\}_n$ and a vector $\varepsilon = (\varepsilon_n)_n$. To be specific, the intervals $I_n$ satisfy $\lambda(\bigcup I_n)=1$ and $\lambda(I_n \cap I_k ) = 0$ if $n \neq k$ and on $I_n$ the map $T_{P, s}$ is given by
\[ T_{\mathcal P, \varepsilon }(x) = \varepsilon_n \frac{r_n-x}{\ell_n-r_n} + (1-\varepsilon_n) \frac{x-\ell_n}{\ell_n-r_n},\]
and $T_{\mathcal P, \varepsilon}(x)=0$ if $x \in [0,1]\setminus \bigcup I_n$. Any GLS transformation produces by iteration number expansions similar to the L\"uroth expansions from \eqref{d:Lexp}, called {\em generalised L\"uroth expansions}. In \cite{BaBu,DK96} the authors studied the approximation coefficients 
\[\theta_n^{\mathcal P_L, \varepsilon} = q_n \big|x - \frac{p_n}{q_n}\big|, \quad n > 0\]
for GLS maps with the standard L\"uroth partition $\mathcal P_L = \big\{ \big[ \frac1{n}, \frac{1}{n-1} \big) \big\}_{n \ge 2} =\{ [z_n, z_{n-1})\}_{n \geq 2}$. In particular \cite[Corollary 2]{BaBu} gives that amongst all GLS systems with standard L\"uroth partition, $T_A$ presents the best approximation properties. More precisely, it states that for any such GLS map $T_{\mathcal P_L, \varepsilon}$ there exists a constant $M_{\mathcal P_L, \varepsilon}$ such that for Lebesgue a.e.~$x \in [0,1]$, the limit
\begin{equation}\label{t:approxcon}
\lim_{n \rightarrow \infty} \frac1n \sum_{i=1}^n \theta_i^{\mathcal P_L, \varepsilon} (x) 
\end{equation}
exists and equals $M_{\mathcal P_L, \varepsilon}$. Furthermore, $M_A \leq M_{\mathcal P_L, \varepsilon} \leq M_L$. In Remark 2, below Corollary 2, of \cite{BaBu}, the authors remark that not every value in the interval $[M_A, M_L]$ can be obtained by such GLS maps, and they suggest furthermore that the achievable values of $M_{\mathcal P_L, \varepsilon}$ might form a fractal set. The situation changes for the random system $L$, as the next theorem shows.

\vskip .2cm
For the random L\"uroth map $L=L_0$ we define for each $(\omega, x) \in \{0,1\}^{\mathbb{N}} \times [0,1]$ and $n \ge 1$ the {\em $n$-th approximation coefficient} by
\[ \theta_n(\omega,x) = q_n \Big| x - \frac{p_n}{q_n} \Big|,\]
where $q_n = (d_n-s_n) \prod_{i=1}^{n-1} d_i(d_i-1)$.

\begin{theorem}\label{p:ethetan}
For the random L\"uroth map $L=L_0$ and $0<p <1$ the limit
\begin{equation}\nonumber
\lim_{n \rightarrow \infty} \frac1n \sum_{i=1}^n \theta_i (\omega, x) 
\end{equation}
exists for $m_p \times \lambda$-a.e.~$(\omega, x) \in \{0,1\}^{\mathbb{N}} \times [0,1]$ and equals
\[M_p := \mathbb{E}[\theta_n] = p \frac{2 \zeta(2)-3 }{2} + \frac{2-\zeta(2)}{2}, \]
where $\zeta(2)$ is the zeta function at $2$. The function $p \mapsto M_p$ maps the interval $[0,1]$ onto the interval $[M_{A}, M_{L}]$. 
\end{theorem}

\begin{proof}
Fix $(\omega, x) \in \{0,1\}^{\mathbb{N}} \times [0,1]$. Let $\theta_n=\theta_n (\omega, x)$, $d_n=d_n(\omega, x) $ and $x_n= T_\omega^n (x)$. By definition of the convergents $\frac{p_n}{q_n}$ we have
\[ \abs{x - \frac{p_n}{q_n}} = \frac{x_n}{\prod_{i=1}^n d_i(d_i-1)}\] 
and the numbers $q_n$ are such that we can write
\begin{equation}\label{e:thetan}
\theta_n =\begin{cases}
\dfrac{x_n}{d_n-1} = d_n x_{n-1} -1, & \mbox{if } x_n=T_L (x_{n-1}),\\[10pt]
\dfrac{x_n}{d_n}=-(d_n -1)x_{n-1} +1,  & \mbox{if } x_n=T_A (x_{n-1}).
 \end{cases}
\end{equation}

Since the Lebesgue measure is stationary for $L_0$, it follows from Birkhoff's Ergodic Theorem that the sequence $(x_n)_n$ is uniformly distributed over the interval $[0,1]$. % By \eqref{e:thetan}, which relates $x_n$ to $\theta_n$, also the approximation coefficients $\theta_n$'s are uniformly distributed over the interval $[0,1]$.
We use this fact together with \eqref{e:thetan} to compute the cumulative distribution function (CDF) $F_{\theta_n}$ of $\theta_n$. For $\mathbb{P}= m_p \times \lambda$, by definition of the CDF and by the law of total probability, for $y \in [0, 1 ]$ we have
\begin{equation}\nonumber
\begin{split}
F_{\theta_n}(y)=\mathbb{P}(\theta_n \leq y)= \, & \mathbb{P}(\theta_n \leq y \, | \, x_n=T_L (x_{n-1}) ) \mathbb{P}(x_n=T_L (x_{n-1})) + \\[10pt]
& \mathbb{P}(\theta_n \leq y \, | \, x_n=T_A (x_{n-1}) ) \mathbb{P}(x_n=T_A (x_{n-1})) \\[10pt]
= & \sum_{d=2}^{\infty} \mathbb{P}\bigg( x_{n-1} \leq \frac{1+y}{d}, \quad  x_{n-1} \in [ z_d, z_{d-1}) \bigg) \mathbb{P}(x_n=T_L (x_{n-1}))  +\\[10pt]
&  \ \mathbb{P}\bigg(  x_{n-1} \geq \frac{1-y}{d-1}, \quad  x_{n-1} \in [ z_d, z_{d-1}) \bigg) \mathbb{P}(x_n=T_A (x_{n-1})) \\[10pt]
= & \sum_{d=2}^{\infty} p \mathbb{P}\bigg( x_{n-1} \leq \frac{1+y}{d}, \quad  x_{n-1} \in [ z_d, z_{d-1}) \bigg)   +\\[10pt]
&  \ (1-p) \mathbb{P}\bigg(  x_{n-1} \geq \frac{1-y}{d-1}, \quad  x_{n-1} \in [ z_d, z_{d-1}) \bigg) .
\end{split} 
\end{equation}
Let $d \geq 2$ be such that $x_{n-1}\in \big[ \frac1d, \frac{1}{d-1} \big)= [z_d,z_{d-1})$, then either $ x_n=T_L (x_{n-1}) \ \text{or} \ x_n=T_A (x_{n-1})$. In the first case, since $d \in \mathbb{N}$, 
\[\frac{1+y}{d} \in \bigg[\frac1d, \frac{1}{d-1} \bigg) \quad \text{if and only if } \quad d \leq \left\lfloor\dfrac{1}{y}\right\rfloor +1,\]
and we obtain
\begin{equation}\nonumber 
\mathbb{P} \bigg( x_{n-1} \leq \frac{1+y}{d} ,  \quad x_{n-1} \in [ z_d, z_{d-1}) \bigg) = \begin{cases}
\dfrac{y}{d}, & \mbox{if } d \leq \left\lfloor\dfrac{1}{y}\right\rfloor +1, \\[10pt]
\dfrac{1}{d(d-1)}, & \mbox{otherwise}.
\end{cases}
\end{equation}
Similarly, in the latter 
\[ \frac{1-y}{d-1} \in  \bigg[\frac1d, \frac{1}{d-1} \bigg) \quad \text{if and only if } \quad d \leq \left\lfloor \frac{1}{y} \right\rfloor,\]
which gives 
\begin{equation}\nonumber 
\mathbb{P} \bigg( x_{n-1} \geq \frac{1-y}{d-1}  ,  \quad x_{n-1} \in [ z_d, z_{d-1}) \bigg) = \begin{cases}
\dfrac{y}{d-1}, & \mbox{if } d \leq \left\lfloor\dfrac{1}{y}\right\rfloor , \\[10pt]
\dfrac{1}{d(d-1)}, & \mbox{otherwise}.
\end{cases}
\end{equation}
Note that 
\[\sum_{d > \left\lfloor 1/y \right\rfloor +1 } \dfrac{1}{d(d-1)} =  \frac{1}{\left\lfloor 1/y \right\rfloor +1} \quad \text{and} \quad \sum_{d > \left\lfloor 1/y \right\rfloor } \dfrac{1}{d(d-1)} =  \frac{1}{\left\lfloor 1/y \right\rfloor}.  \]
Summing over all $d \geq 2$ gives
\[ F_{\theta_n}(y)= p \bigg( \sum_{d=2}^{\left\lfloor 1/y \right\rfloor +1 } \frac{y}{d} + \frac{1}{\left\lfloor 1/y \right\rfloor +1} \bigg) + (1-p) \bigg (\sum_{d=2}^{\left\lfloor 1/y \right\rfloor } \frac{y}{d-1} + \frac{1}{\left\lfloor 1/y \right\rfloor} \bigg),\]
so that by \eqref{q:fl} and Theorem~\ref{t:fafl} we obtain
\[F_{\theta_n}(y)= p F_L(y) + (1-p) F_A(y).\]
%Remark 2 below Theorem 4 of \cite{BaBu} gives, for the cumulative distribution functions $F_L$ and $F_A$ of the standard and alternating L\"uroth maps respectively, the following equalities
%\[\int_{0}^{1}  F_L (y) dy = \frac32 -\frac{ \zeta(2)}{2} \quad \text{and } \quad \int_{0}^{1}  F_A(y) dy = \frac{ \zeta(2)}{2}.\]
The expectation $\mathbb{E}[\theta_n]$ can be now computed by $\mathbb{E}[\theta_n] = \int_{0}^{1} (1-F_{\theta_n}(y)) dy,$ which with the results from Theorem~\ref{t:fafl} gives
\[\begin{split}
\mathbb{E}[\theta_n] = & 1- p \int_{0}^{1} F_L(y) dy - (1-p) \int_{0}^{1} F_A(y) dy\\
= & 1-p \bigg( \frac32 -\frac{ \zeta(2)}{2} \bigg) - (1-p) \bigg(\frac{ \zeta(2)}{2} \bigg)\\
= & p \frac{2 \zeta(2)-3 }{2} + \frac{2-\zeta(2)}{2},
\end{split} \]
that is
\begin{equation}\label{e:mppoly}
\lim_{n \rightarrow \infty} \frac1n \sum_{i=1}^n \theta_i (\omega, x) = p \frac{2 \zeta(2)-3 }{2} + \frac{2-\zeta(2)}{2} = M_p.
\end{equation}
Note that for $p=0$, $x_n= T_A^n$ for $n \geq 0$ and indeed $M_0=M_{A}$. In the same way, for $p=1$, $x_n= T_L^n$ for $n \geq 0$ and $M_1=M_{L}$. Lastly, $M_p$ is an increasing function in $p \in [0,1]$, so it can assume any value between $M_{A}$ and $M_{L}$.  
\end{proof}

Theorem \ref{t:main2} is now given by Proposition \ref{p:universal0}, Proposition \ref{p:le} and Theorem \ref{p:ethetan}.

\section{Generalised L\"uroth expansions with bounded digits}\label{s:finitea}

Both Proposition~\ref{p:le} and Theorem~\ref{p:ethetan} depend on the invariant measure $m_p \times \lambda$ for $L_0$. In Proposition~\ref{p:le} it is used to compute the value of $\Lambda(\omega,x)$ and in Theorem~\ref{p:ethetan} we use the fact that the numbers $T^n_\omega(x)$ are uniformly distributed over the interval $[0,1]$ and it is used to deduce \eqref{e:mppoly}. For $c  >0$, the random map $L_c$ becomes fundamentally different: The maps $T_{j,c}$ for $j=0,1$ present finitely many branches that are not always onto. These variations make the measures $m_p \times \lambda$ no longer $L_c$-invariant. It still follows from results in e.g.~\cite{Mo, Pe, GoBo,In} that for any $0<p<1$ the random map $L_c$ admits an invariant probability measure of type $m_p \times \mu_{p,c}$, where $m_p$ is the $(p,1-p)$-Bernoulli measure on $\{0,1\}^{\mathbb{N}}$ and $\mu_{p,c} \ll \lambda$ is a probability measure on $[c,1]$ that satisfies
\begin{equation}\label{q:pfmeasure}
\mu_{p,c} (B) = p \mu_{p,c} (T_{0,c}^{-1}(B)) + (1-p) \mu_{p,c} (T_{1,c}^{-1}(B))
\end{equation}
for each Borel measurable set $B \subseteq [c,1]$. We set $f_{p,c}:= \frac{\de \mu_{p,c}}{\de \lambda}$. Here we call $\mu_{p,c}$ a \emph{stationary measure} and $f_{p,c}$ an \emph{invariant density} for $L_c$. With \cite[Corollary 7]{Pe} it follows that since $T_{0,c}$ is expanding and has a unique absolutely continuous invariant measure, the measure $\mu_{p,c}$ is the unique stationary measure for $L_c$ and that $L_c$ is ergodic with respect to $m_p \times \mu_{p,c}$.

\begin{example}\label{e:L13}
For $c=\frac13$ and $0 < p < 1$ consider the measure $\mu_{p,\frac13}$ with density
\begin{equation}\nonumber
f_{p, \frac13}(x)= \begin{cases}
\frac{9}{8}, & \text{if  } x \in \big[\frac13, \frac23\big), \\[10pt]
\frac{15}{8}, & \text{if  } x \in \big[\frac23, 1].
\end{cases}
\end{equation}
One can check by direct computation that $\mu_{p, \frac13}$ satisfies \eqref{q:pfmeasure} and thus is the unique stationary measure for $L_{\frac13,p}$.
\end{example}

Since such an invariant measure exists, some of the results from the previous section also hold for the maps $L_{c,p}$ with $c>0$. In particular the Lyapunov exponent from \eqref{d:Lyap} is well defined and since the stationary measure $m_p \times \mu_{p,c}$ is still ergodic, we can apply Birkhoff's Ergodic Theorem to obtain for $(m_p \times \mu_{p,c})$-a.e.~point $(\omega, x)$ the following expression for the Lyapunov exponent:
\begin{equation}\label{d:specificlyap}
\Lambda(\omega, x) = \sum_{d=2}^{d_c} \mu_{p,c} \bigg( \bigg( \frac{1}{d}, \frac{1}{d-1} \bigg)\bigg)  \log (d(d-1))  + \mu_{p,c} \bigg( \bigg(  c, \frac{1}{d_c} \bigg) \bigg) \log (d_c(d_c+1)) ,
\end{equation}
where $d_c$ is the unique positive integer such that $c \in \big[ \frac{1}{d_c+1}, \frac{1}{d_c} \big)$.

\begin{example}\label{e:L13more}
Consider the map $L_{\frac13,p}$ from Example \ref{e:L13} again. The possible digits of the generalised L\"uroth expansions produced by $L_{\frac13,p}$ are $(0,2), \, (1,2), \, (0,3), \, (1,3)$. It follows from Birkhoff's Ergodic Theorem that the frequency of the digit $(0,2)$ is given by
\begin{equation}\nonumber 
\begin{split}
\pi_{(0,2)} &= \lim_{n \to \infty} \frac{ \# \{ 1 \leq j \leq n: \, s_j(\omega,x)=0 \, \text{ and } \, d_j(\omega, x)=2\}}{n} \\
& = \int_{\{0,1\}^{\mathbb{N}} \times [\frac13,1]} \mathbf{1}_{[0] \times [\frac23, \frac56]} + \mathbf{1}_{\{0,1\}^{\mathbb{N}} \times (\frac56,1]} \text{d} m_{p} \times \mu_{p, \frac13} \\
& = \frac{15}{8}p \bigg( \frac56 - \frac23 \bigg) + \frac{15}{8}\bigg( 1 - \frac56 \bigg)= \frac{5+5p}{16}.
\end{split}
\end{equation}
Similarly,
\[\pi_{(1,2)}= \frac{8-5p}{16}, \qquad \pi_{(0,1)}= \frac{1+p}{16}, \qquad \pi_{(1,3)}= \frac{2-p}{16}.\]
Note that for 
\[\pi_2=  \lim_{n \to \infty} \frac{ \# \{ 1 \leq j \leq n: \,  d_j(\omega, x)=2\}}{n} = \pi_{(0,2)}+ \pi_{(1,2)}\] 
we also obtain 
\[\pi_2= \mu_{p, \frac13} \bigg(\bigg[ \frac12, 1 \bigg]\bigg)= \frac{13}{16}, \quad \text{and} \quad \pi_3= 1-\pi_2= \frac{3}{16}.\]
Moreover for $m_p \times \mu_{p,\frac13}$-a.e. $(\omega, x)$ we have by \eqref{d:specificlyap} that 
\begin{equation}\nonumber
\begin{split}
\Lambda(\omega, x)  =\ & \sum_{d=2}^{3} \mu_{p,\frac13} \bigg( \bigg[ \frac{1}{d}, \frac{1}{d-1} \bigg)\bigg) \log (d(d-1))  \\
 =\ &  \mu_{p, \frac13} \bigg( \bigg[ \frac12, 1 \bigg] \bigg) \log 2 + \mu_{p, \frac13} \bigg( \bigg[ \frac13, \frac12 \bigg) \bigg) \log 6   \\
=\ & \frac{13}{16} \log 2   +  \frac{3}{16} \log 6  = 0.89913 \ldots < 1.198328 \ldots = \Lambda_{m_p \times \lambda} . 
\end{split}
\end{equation}
\end{example}

From \eqref{d:specificlyap} and Example \ref{e:L13more} it is clear that to obtain results similar to Proposition~\ref{p:le} for $c>0$ we need a good expression for the density of $\mu_{p, c}$ and an accurate description of the location of the points $\frac{p_n}{q_n}$ relative to $x$. %For $c=0$ we saw that it immediately follows that Lebesgue almost all $x \in [0,1]$ have uncountably many different $0$-L\"uroth expansions.
We start this section by giving some results on the number of different $c$-L\"uroth expansions a number $x$ can have for $c>0$.

\subsection{Uncountably many universal expansions}
Let $c \in \big(0, \frac12\big]$ and consider the corresponding alphabet $\mathcal A_c = \{ (i,j) \, : \, i \in \{0,1\}, \, j\in \{ 2,3, \ldots ,\lceil \frac1c \rceil \} \}$. We call a sequence $((s_n,d_n))_{n \ge 1} \in \mathcal A_c^\mathbb N$ {\em $c$-L\"uroth admissible} if
\[ \sum_{n \ge 1} (-1)^{\sum_{i=1}^{n-1}s_{i+k}} \frac{d_{n+k}-1+s_{n+k}}{\prod_{i=1}^n d_{i+k}(d_{i+k}-1)} \in [c,1]\]
for all $k \ge 0$ and if moreover the sequence does not end in $((0,d+1),(0,2)^\infty)$ for some $2 \leq d < \lceil \frac1c \rceil$.

\vskip .2cm
The first main result of the number of different $c$-L\"uroth expansions a number $x \in [c,1]$ can have is the following. 
\begin{theorem}\label{t:uncountable25}
Let $0<c \leq \frac25$. Then every $x \in [c,1] \setminus \mathbb Q$ has uncountably many different $c$-L\"uroth expansions.
\end{theorem}

Before we prove the theorem we prove three lemmata. 
\begin{lemma}\label{l:rightdigits}
Let $c \in \big(0, \frac12\big)$ and let $((s_n,d_n))_{n \ge 1} \in \mathcal A_c^\mathbb N$ be $c$-L\"uroth admissible. For
\[ x = \sum_{n \ge 1} (-1)^{\sum_{i=1}^{n-1}s_i} \frac{d_n-1+s_n}{\prod_{i=1}^n d_i(d_i-1)}\]
the following hold.
\begin{itemize}
\item If $s_1 =0$, then $x \in \big[ \frac1{d_1}+ \frac{c}{d_1(d_1-1)} , \frac1{d_1-1} \big]=[z_{d_1}^+, z_{d_1-1}]$.
\item If $s_1=1$, then $x \in \big[\frac1{d_1}, \frac1{d_1-1}-\frac{c}{d_1(d_1-1)}\big] = [z_{d_1}, z_{d_1-1}^-]$.
\end{itemize}
\end{lemma}

\begin{proof}
For $x$ we can write
\[ x = \frac{d_1-1+s_1}{d_1(d_1-1)} + \frac{(-1)^{s_1}}{d_1(d_1-1)} \sum_{n \ge 1} (-1)^{\sum_{i=1}^{n-1}s_{i+1}} \frac{d_{n+1}-1+s_{n+1}}{\prod_{i=1}^{n} d_{i+1}(d_{i+1}-1)}.\]
Write
\[ x_1= \sum_{n \ge 1} (-1)^{\sum_{i=1}^{n-1}s_{i+1}} \frac{d_{n+1}-1+s_{n+1}}{\prod_{i=1}^{n} d_{i+1}(d_{i+1}-1)}.\]
Since $((s_n,d_n))_{n \ge 1}$ is $c$-L\"uroth admissible, we have $x_1 \in [c,1]$. If $s_1=0$, then
\[ \frac{d_1-1+s_1}{d_1(d_1-1)} + \frac{(-1)^{s_1}}{d_1(d_1-1)} x_1 = \frac1{d_1} + \frac1{d_1(d_1-1)} x_1\in \bigg[ \frac1{d_1} + \frac{c}{d_1(d_1-1)} , \frac1{d_1-1} \bigg].\]
Similarly if $s_1=1$, then
\[ \frac{d_1-1+s_1}{d_1(d_1-1)} + \frac{(-1)^{s_1}}{d_1(d_1-1)} x_1 = \frac1{d_1-1} - \frac1{d_1(d_1-1)} x_1\in \bigg[ \frac1{d_1}, \frac1{d_1-1} - \frac{c}{d_1(d_1-1)} \bigg]. \qedhere\]
\end{proof}

Recall the definition of the switch region $S$ from \eqref{d:switch}, that is
\[S= [c,1] \cap \bigcup_{n >1}  [z_n^+, z_{n-1}^-].\]

\begin{lemma}\label{l:admissible}
A sequence $((s_n,d_n))_{n \ge 1} \in \mathcal A_c^\mathbb N$ is $c$-L\"uroth admissible if and only if there is a point $(\omega,x) \in \{0,1\}^\mathbb N \times [c,1]$ such that the sequence $(s_n(\omega,x),d_n(\omega,x))_{n \ge 1}$ generated by $L_c$ satisfies $ s_n(\omega,x)=s_n$ and $d_n(\omega,x)=d_n$ for each $n \ge 1$.
\end{lemma}

\begin{proof}
One direction follows since $L_c^k(\omega,x) \in [c,1]$ for all $k$ and by the definition of $L_c$, $s_i$ and $d_i$ on the points $z_n$. For the other direction, let $((s_n,d_n))_n$ be $c$-L\"uroth admissible. Set
\[ x= \sum_{n \ge 1} (-1)^{\sum_{i=1}^{n-1}s_i} \frac{d_n-1+s_n}{\prod_{i=1}^n d_i(d_i-1)}.\]
For $k \ge 0$ define the numbers
\[ x_k =  \sum_{n \ge 1} (-1)^{\sum_{i=1}^{n-1}s_{i+k}} \frac{d_{n+k}-1+s_{n+k}}{\prod_{i=1}^n d_{i+k}(d_{i+k}-1)}, \]
so that $x_0=x$. Let $\omega \in \{0,1\}^\mathbb N$ be such that
\[ \omega_k = \begin{cases}
0, & \text{ if } 1_S(x_{k-1})=1 \, \text{ and }\, s_k=0,\\
1, & \text{ otherwise}. 
\end{cases} \]
By Lemma~\ref{l:rightdigits} and the fact that $((s_n,d_n))_n$ does not end in $((0,d+1),(0,2)^\infty)$ it then follows that $s_1(\omega,x)=s_1$ and $d_1(\omega,x)=d_1$. Moreover, $T_\omega^1(x) = x_1$. It then follows by induction that for each $n \ge 1$, $s_n = s_n(\omega,x)$, $d_n= d_n(\omega,x)$ and $T_\omega^n(x)=x_n$. Hence, $((s_n,d_n))_n$ corresponds to the $c$-L\"uroth expansion of $(\omega,x)$.
\end{proof}

\begin{lemma}\label{l:allinS}
Let $0 < c \leq \frac25$. For any $x \in [c,1] \setminus \mathbb Q$ and any $\omega \in \{0,1\}^\mathbb N$ there is an $n \ge 0$, such that $T_\omega^n (x) \in S$.
\end{lemma}

\begin{proof}
Let $0 < c \le \frac25$, $x \in  [c,1] \setminus \mathbb Q$ and $\omega \in \{0,1\}^\mathbb N$. If $x \in S$, then we are done. If $x \not \in S$, then $T_{\omega,c}^1(x)=T_{0,c}(x) = T_{1,c}(x) \in (1-c,1)$. Assume that $c \le \frac13$, then $z_2^+=\frac{1+c}{2} \le 1-c < 1-\frac{c}{2}=z_1^-$. Let $f: x \mapsto 2x-1$ denote the right most branch of $T_L$. Since $f(z_1^-)=1-c$, it follows that $(1-c, 1) = \cup_{j \in \mathbb{N}} f^{-j} ( (1-c, z_1^-])$ and hence there exists a $j \in \mathbb{N}$ such that $T_{\omega,c}^{j+1}(x) = T_L^j(T_{\omega,c}^1(x))  \in (1-c, z_1^-] \subseteq S$.

\vskip .2cm
Now assume that $\frac13 < c \leq \frac25$. Then $\frac12 < 1-c < z_2^+$ and
\[z_2^+ < 2c= T_{0,c}(1-c)=T_{1,c}(1-c) < z_1^-.\]
We write 
\[ (1-c,1) = ( 1-c, z_2^+) \cup [z_2^+, z_1^-] \cup (z_1^-, 1), \]
and we treat the subintervals separately.

\begin{itemize}
\item[1.] For the first subinterval note that $( 1-c, z_2^+) = \big(1-c, \frac{3-c}{4}\big] \cup \big(\frac{3-c}{4}, \frac23\big) \cup \big[\frac23, z_2^+\big)$. Here $\frac23$ is a repelling fixed point of $T_A$ and this subdivision is such that $T_A \big(\big(1-c, \frac{3-c}{4}\big)\big) = [z_2^+, 2c) \subseteq S$, $T_A \big(\big(\frac{3-c}{4}, \frac23\big)\big) = \big( \frac23, z_2^+ \big)$ and $T_A \big( \big(\frac23, z_2^+ \big) \big) = \big(1-c, \frac23 \big)$. This gives the following.\\

- If $T_{\omega,c}^1(x) \in  \big(1-c, \frac{3-c}{4}\big]$, then $T_{\omega,c}^2(x) \in [z_2^+, 2c) \subseteq S$.\\

- If $T_{\omega,c}^1(x) \in  \big(\frac{3-c}{4}, \frac23\big)$, then $T_{\omega,c}^2(x) = T_A (T_{\omega,c}^1(x))$ and since $\frac23$ is a repelling fixed point for $T_A$ there must then exist a $j$ such that $T_{\omega,c}^j (x) \in \big(1-c, \frac{3-c}{4} \big)$, so $T_{\omega,c}^{j+1} (x) \in (z_2^+,2c) \subseteq S$.\\

%note that $T_A ([\frac23, z_2^+]) = [1-c, \frac{3-c}{4}] \cup [\frac{3-c}{4}, \frac23]$ and $T_A ([\frac{3-c}{4}, \frac23]) = [\frac23, z_2^+]$. It also holds that for any irrational $y_0 \in [\frac{3-c}{4}, \frac23]$, then
%\[T_A^2 (y_0) < y_1 < T_A(y),\]
%since $4y-2 < y < -2y + 2$ holds for any $y < \frac23$. Applying the same equality to $y_2 = T^2_A(y_0)$ and in general to $y_{2n}= T^{2n}_A(y_0)$, we obtain a sequence of strictly decreasing points 
%\[\{y_{2n} = (-2)^{2n} y + \sum_{k=1}^{2n} 2 \cdot (-2)^{k-1} \}_{2n},\] 
%such that $\lim_{n \rightarrow \infty} y_{2n} =- \infty$. This implies that there exists $j= j(y_0)$ such that $y_{j}= T_A^{j} (y) > \frac{3-c}{4}$ and as such $T_A^{j} (y) \in [1-c, \frac{3-c}{4}]$. If $T_\omega^1(x) \in [\frac{3-c}{4}, \frac23]$, then for $t= j(T_\omega^1(x)) + 2$, $T_\omega^t(x) \in [z_2^+, 2c] \subseteq S$.\\

- If $T_{\omega,c}^1(x) \in \big[\frac23, z_2^+\big)$, then either $T_{\omega,c}^2(x) \in \big(1-c, \frac{3-c}{4}\big]$ and $T_{\omega,c}^3(x) \in [z_2^+, 2c) \subseteq S$; or $T_{\omega,c}^2(x) \in \big(\frac{3-c}{4}, \frac23\big]$, and again we can find a suitable $j$ as above.
\item[2.] If $T_{\omega,c}^1(x) \in [z_2^+, z_1^-]$, then $T_{\omega,c}^1(x) \in S$. 
\item[3.] If $T_{\omega,c}^1(x) \in (z_1^-, 1]$, since $f(z_1^-)= 1-c$ we can write $(z^-_1,1)$ as the disjoint union
\[ (z_1^-, 1) = \cup_{j \ge 1} f^{-j} ([z_2^+,z_1^-]) \cup f^{-j} ((1-c, z_2^+)).\]
Hence, there is a $j$ such that either $T_{\omega,c}^{j+1} (x) = T_L^j \circ T_{\omega,c}^1 (x) \in S$ and we are done or $T_{\omega,c}^{j+1} (x) = T_L^j \circ T_{\omega,c}^1 (x) \in (1-c,z_2^+)$ and we are in the situation of case 1.
\end{itemize}
This finishes the proof.
\end{proof}

\begin{proof}[Proof of Theorem~\ref{t:uncountable25}]
Let $0 < c \le \frac25$ and $x \in [c,1] \setminus \mathbb Q$ be given. To prove the result it is enough to show that for any sequence $((s_n,d_n))_{n \ge 1}$ representing a $c$-L\"uroth expansion of $x$ and any $N \ge 1$ there is an $n \ge N$ and a $c$-L\"uroth expansion
\[ ((s_1,d_1), \ldots, (s_{N+n}, d_{N+n}), (s_{N+n+1}',d_{N+n+1}'), (s_{N+n+2}', d_{N+n+2}'), \ldots)\]
of $x$ with $s_{N+n+1}' \neq s_{N+n+1}$ or $d_{N+n+1}' \neq d_{N+n+1}$. 

\vskip .2cm
Let $((s_n,d_n))_{n \ge 1}$ be any $c$-L\"uroth admissible sequence with
\[ x = \sum_{n \ge 1} (-1)^{\sum_{i=1}^{n-1}s_i} \frac{d_n-1+s_n}{\prod_{i=1}^n d_i(d_i-1)}.\]
By Lemma~\ref{l:admissible} there then exists a sequence $\omega$, such that $(s_n(\omega,x),d_n(\omega,x))_n = ((s_n,d_n))_n$. Fix an $N \ge 1$. Lemma~\ref{l:allinS} yields the existence of an $n$ such that $T_{\omega,c}^{N+n}(x) \in S$. Take any sequence $\tilde \omega \in \{0,1\}^\mathbb N$ with $\tilde \omega_j = \omega_j$ for all $1 \le j \le N+n$, $\tilde \omega_{N+n+1} = 1-\omega_{N+n+1}$. Then for each $1 \le j \le N+n$,
\[ s_j(\tilde \omega,x) = s_j(\omega,x) = s_j \quad \text{and} \quad d_j(\tilde \omega,x) = d_j(\omega,x) = d_j\]
and
\[ s_{N+n+1} (\tilde \omega,x) = \tilde \omega_{N+n+1} = 1- \omega_{N+n+1} = 1- s_{N+n+1} (\omega,x).\]
From $T_{\tilde \omega,c}^{N+n+1}(x) \not \in \mathbb Q$ we obtain that either $T_{\tilde \omega,c}^{N+n+1}(x) > \frac12$ and $T_{\omega,c}^{N+n+1}(x) < \frac12$ so that $d_{N+n+2} (\tilde \omega,x) =2$ and $d_{N+n+2} (\omega,x) >2$, or $T_{\tilde \omega,c}^{N+n+1}(x) < \frac12$ and $T_{\omega,c}^{N+n+1}(x) > \frac12$ so that $d_{N+n+2} (\tilde \omega,x) >2$ and $d_{N+n+2} (\omega,x) =2$. In any case we get a $c$-L\"uroth expansion of $x$ that differs from the expansion $((s_n,d_n))_n$ at indices $N+n+1$ and $N+n+2$.
\end{proof}

The second result of this section is on universal expansions. For the remainder of this section we assume that $c = \frac1\ell$ for some $\ell \ge 3$. Then $\mathcal A_c = \{ (i,j) \, : \, i \in \{0,1\}, \, j \in \{2,3, \ldots, \ell\} \}$ and by Lemma~\ref{l:admissible} any sequence in $\mathcal A_c^\mathbb N$ not ending in $((0,d+1),(0,2)^\infty)$ is $c$-L\"uroth admissible. An expansion
\[ x = \sum_{n \ge 1} (-1)^{\sum_{i=1}^{n-1}s_i} \frac{d_n-1+s_n}{\prod_{i=1}^n d_i(d_i-1)}\]
of a number $x \in [c,1]$ is called a {\em universal} $c$-L\"uroth expansion if for any block $(t_1,b_1), \ldots, (t_j,b_j) \in \mathcal A_c$ there is a $k \ge 1$, such that $s_{k+i} = t_i$ and $d_{k+i} = b_i$ for all $1 \le i \le j$, i.e., if each finite block of digits occurs in $((s_n,d_n))_n$. 

\begin{theorem}\label{t:universal}
Let $c = \frac1\ell$ for some $\ell \in \mathbb{N}_{\ge 3}$. Then Lebesgue almost every $x \in \big[\frac1{\ell},1\big]$ has uncountably many different universal $c$-L\"uroth expansions.
\end{theorem}

The proof of this theorem requires some work and several smaller results. First we prove a property of the measure $m_p \times \mu_{p,c}$.
\begin{proposition}\label{p:equivm}
Let $c = \frac1{\ell}$ for some $\ell \in \mathbb{N}_{\ge 3}$. Then for any $0 < p < 1$ the random transformation $L_c$ is mixing and the density of $\mu_{p,c}$ is bounded away from 0.
\end{proposition}

\begin{proof}
We will show that $L_c$ has the {\em random covering property}, i.e., that for any non-trivial subinterval $J\subseteq [c,1]$ there is an $n \ge 1$ and an $\omega \in \{0,1\}^\mathbb N$ such that $T_\omega^n (J) = [c,1]$. The result then follows from \cite[Proposition 2.6]{ANV15}.

\vskip .2cm
Let $J \subseteq [c,1]$ be any interval of positive Lebesgue measure. Since $|\frac{\de}{\de x} T_{\omega,c}^1 (x)| \ge 2$, for any $\omega \in \{0,1\}^\mathbb N$ there is an $m$ such that at least one of the points $ z_n, z_n^+, z_n^-$ is in the interior of $T_\omega^m (J)$ and hence $T_\omega^{m+2}(J)$ will contain an interval of the form $(a,1]$ for some $a$. Since 1 is a fixed point, this implies that there is a $k$ such that $(1-\frac1{\ell}, 1] \subseteq T_\omega^k ((a,1])$ for each $\omega \in \{0,1\}^\mathbb N$. %For each $t \ge 0$ we have $T_0^t \big( 1- \frac1{\ell} \big)= 1 -\frac{2^t}{\ell}$. Hence,
For the smallest $i \ge \log_2(\ell-1)-1$ it holds that $T_{0^i,c}\big( 1- \frac1{\ell} \big) = 1 -\frac{2^i}{\ell} < \frac12 + \frac{1}{2\ell}= z_2^+$. Hence
\[ [c,1] \subseteq T_{0^{i+1},c} \Big(\Big(1-\frac1{\ell}, 1\Big]\Big),\]
giving the random covering property and the result.
\end{proof}

The difference between Theorem~\ref{t:universal} and Proposition~\ref{p:universal0} is that in case $c=0$ for almost all $(\omega,x)$ the sequence $(s_n(\omega,x))_n$ equals the sequence $\omega$, so that it is immediately clear that different sequences $\omega$ lead to different expansions. For $c>0$ typically many sequences $\omega$ lead to the same sequence $(s_n(\omega,x))_n$. In other words, the map $g: \{0,1\}^\mathbb N \times [c,1] \to \mathcal A_c^\mathbb N$ defined by 
\begin{equation}\label{q:g}
g((\omega,x)) = \big( (s_1(\omega,x), d_1(\omega,x)), (s_2(\omega,x),d_2(\omega,x)), (s_3(\omega,x),d_3(\omega,x)), \ldots)
\end{equation}
is far from injective. To solve this issue we define another random dynamical system $K_c: \{0,1\}^\mathbb N \times [c,1] \to \{0,1\}^\mathbb N \times [c,1]$ given by
\[ K_c (\omega,x) = \begin{cases}
(\omega, T_A (x)), & \text{ if } x \in \cup_{n=2}^{\ell} [ z_n, z_n^+),\\
(\omega,T_L (x)), & \text{ if } x \in   \cup_{n=2}^{\ell} ( z_{n-1}^-, z_{n-1})\cup \{1\},\\
(\sigma(\omega),T_{\omega,c}^1(x)), & \text{ if } x \in S.
\end{cases}\]
The difference between $K_c$ and $L_c$ is that the map $K_c$ only shifts in the first coordinate if the point $x$ lies in $S$, so only if $T_{0,c}(x) \neq T_{1,c}(x)$. For $K_c$ define the function $h: \{0,1\}^\mathbb N \times [c,1] \to \mathcal A_c^\mathbb N$ by 
\[ h( (\omega,x) ) = \big( ( s_1(\omega,x), d_1(\omega,x)), (s_1(K_c(\omega,x)), d_1(K_c(\omega,x))), (s_1(K_c^2(\omega,x)),d_1(K_c^2(\omega,x))), \ldots \big).\]
In the other direction set $\psi_c = \psi|_{\mathcal A_c^\mathbb N}$, so that
\[  \psi_c : \mathcal A_c^\mathbb N \to [c,1], \, ((s_n,d_n))_{n \ge 1} \mapsto  \sum_{n \ge 1} (-1)^{\sum_{i=1}^{n-1}s_i} \frac{d_n-1+s_n}{\prod_{i=1}^n d_i(d_i-1)}.\]
Neither $g$ nor $h$ is surjective, since both maps $L_c$ and $K_c$ do not produce sequences ending in $((0,d+1),(0,2)^\infty)$ as we saw before. To solve this and also make $h$ injective, define 
\[\begin{split}
Z_L = \ & \{(\omega, x) \in \{0,1\}^{\mathbb{N}} \times [c,1]: \, L_c^n(\omega, x) \in \{0,1\}^{\mathbb{N}} \times S \, \text{ for infinitely many } n \ge 0\},\\
Z_K=\ & \{(\omega, x) \in \{0,1\}^{\mathbb{N}} \times [c,1]: \, K_c^n(\omega, x) \in \{0,1\}^{\mathbb{N}} \times S \, \text{ for infinitely many } n \ge 0\},\\
D =\ & \{ \mathbf a \in \mathcal A_c^\mathbb N \, : \, \psi_c(\sigma^n \mathbf a) \in S \, \text{ for infinitely many } n\ge 0\}.
\end{split}\]
Then $g: Z_L \to D$ is surjective and $h:Z_K \to D$ is bijective and by Lemma~\ref{l:allinS} $m_p \times \mu_{p,c}(Z_L)=1$. 
%Neither $g$ nor $h$ is injective, since for example
%To see that both $g$ and $h$ are neither surjective nor injective consider the following examples. For any $c$, the sequences $\mathbf{a}_1=((0,2), (0,3), (0,2)^{\infty})$ and $\mathbf{a}_2=((0,2), (1,2), (0,2)^{\infty})$ are such that 
%\[\psi_c (\mathbf{a}_1)=\psi_c (\mathbf{a}_2) =\frac34.\]
%However, there is no $\omega \in \{0,1\}^{\mathbb{N}}$ such that $g (\omega, \frac34) =\mathbf{a}_2$ nor $h (\omega, \frac34) =\mathbf{a}_2$. This is due to the fact that $T_{j,c}\big(\frac34\big) = \frac12$ for any $j=0,1$ and by construction $\frac12$ lies in the partition element $\big(\frac13, \frac12\big]$ but $\frac12 \notin (\frac12,1]$. This implies that $d_1(\omega, \frac12) = 3$ for any $\omega$. Furthermore, note also that 
%\[g\Big( \Big(01^{\infty}, \frac34\Big) \Big)= g\Big( \Big(0^{\infty}, \frac34\Big) \Big)=h \Big( \Big(01^{\infty}, \frac34\Big) \Big)= h\Big( \Big(0^{\infty}, \frac34\Big) \Big)= ((0,2), (1,2), (0,2)^{\infty}),\]
%
%If we set 
%\[Z= \{(\omega, x) \in \{0,1\}^{\mathbb{N}} \times [c,1]: \, K_c^n(\omega, x) \in \{0,1\}^{\mathbb{N}} \times S \, \text{ for infinitely many } n \ge 0\},\]
%and 
%\[ D = \{ \mathbf a \in \mathcal A_c^\mathbb N \, : \, \psi_c(\sigma^n \mathbf a) \in S \, \text{ for infinitely many } n\ge 0\},\]
%then $h:Z \to D$ is bijective.
The proof of Theorem~\ref{t:universal} is based on the proofs of \cite[Theorem 7 and Lemma 4]{DaVr2} and uses the following result on the map $K_c$.

\begin{proposition}\label{l:invergK}
Let $c = \frac1{\ell}$ for some $\ell \in \mathbb{N}_{\ge 3}$ and $0<p<1$. The measure $m_p \times \mu_{p,c}$ is invariant and ergodic for $K_c$.
\end{proposition}

\begin{proof}
First note that $\sigma \circ g = g \circ L_c$. Define a measure $\nu$ on $\mathcal A^\mathbb N_c$ with the $\sigma$-algebra generated by the cylinders by setting $\nu = m_p \times \mu_{p,c} \circ g^{-1}$. Then by Lemma~\ref{l:allinS} we get $\nu(\mathcal A_c^\mathbb N \setminus D)=0$, that is $\nu$ is concentrated on $D$. So $g$ is a factor map and $\sigma$ is ergodic and measure preserving with respect to $\nu$. By construction it holds that $K_c^{-1}(Z_K) = Z_K$ and $\sigma^{-1}(D)=D$ and moreover, $\sigma \circ h = h \circ K_c$. Define a measure $\tilde \nu$ on $\{0,1\}^\mathbb N \times [c,1]$ by setting
\[ \tilde \nu(A) = \nu \big( h(A \cap Z_K)).\]
Since $h$ is a bijection from $Z_K$ to $D$ we find that $h: \{0,1\}^\mathbb N \times [c,1] \to \mathcal A_c^\mathbb N$ is an isomorphism and $K_c$ is measure preserving and ergodic with respect to $\tilde \nu$. What is left is to prove that $\tilde \nu= m_p \times \mu_{p,c}$, which is what we do now.\\

\vskip .2cm
%By definition of $\tilde \nu$,
%\[\tilde \nu (A)= \nu (h (A \cap Z_K))=m_p \times \mu_{p,c} \circ g^{-1} (h (A \cap Z_K)),\]
%so that to show $\tilde \nu= m_p \times \mu_{p,c}$ it is enough to show that $m_p \times \mu_{p,c}$ assigns the same values to the sets $g^{-1} (h (A \cap Z_K))$ and $A$ for arbitrary measurable set $A$. %The cylinders $[(k_1, i_1), \ldots, (k_n, i_n)] \subseteq \mathcal A_c^\mathbb N$ generate the product $\sigma$-algebra on $\mathcal A_c^\mathbb N$ and $h$ is bi-measurable with respect to the product $\sigma$-algebra on $\{0,1\}^\mathbb N \times [c,1]$ of the product $\sigma$-algebra on $\{0,1\}^\mathbb N$ and the Borel $\sigma$-algebra on $[c,1]$. So the sets
Sets of the form
\[ h^{-1} ([(k_1, i_1), \ldots, (k_n, i_n)] ) \]
generate the product $\sigma$-algebra on $\{0,1\}^\mathbb N \times [c,1]$ given by the product $\sigma$-algebra on $\{0,1\}^\mathbb N$ and the Borel $\sigma$-algebra on $[c,1]$. Therefore it is enough to check that
\[ \tilde \nu (h^{-1} ([(k_1, i_1), \ldots, (k_n, i_n)] )) = m_p \times \mu_{p,c} (h^{-1} ([(k_1, i_1), \ldots, (k_n, i_n)] ))\]
for any cylinder $[(k_1, i_1), \ldots, (k_n, i_n)] \subseteq \mathcal A_c^\mathbb N$. 

\vskip .2cm
For $i \in \{2,3,\ldots,\ell\}$, let
\begin{equation}\nonumber
\begin{split}
A_{0i} &= [0] \times \big[ z_i^+, z_{i-1}^-  \big], \qquad A_{22} = \{0,1\}^{\mathbb{N}}  \times \big( z_1^-, 1  \big] \, \, \text{ and } \, \, A_{2i} = \{0,1\}^{\mathbb{N}}  \times \big( z_{i-1}^-, z_{i-1}  \big), \, i\ge3 ,\\
A_{1i} &= [1] \times \big[ z_i^+, z_{i-1}^-  \big], \qquad A_{3i} = \{0,1\}^{\mathbb{N}} \times \big[ z_i, z_{i}^+  \big).\\
\end{split}
\end{equation}
For any cylinder $[(k,i)] \subseteq \mathcal A_c^\mathbb N$ we get $h^{-1}([(k,i)])= A_{k i} \cup A_{(k+2) i}$ and
\begin{equation} \label{q:h-1union}
h^{-1} ([(k_1,i_1),(k_1,i_2),\ldots ,(k_n,i_n)]) = \bigcup_{j_1, \ldots, j_n} A_{j_1 i_1} \cap K_c^{-1}(A_{j_2 i_2}) \cap \cdots \cap K_c^{-(n-1)}(A_{j_n i_n}),
\end{equation}
where the union is disjoint and is taken over all blocks $j_1, \ldots, j_n$ that have $j_t \in \{k_t, k_{t}+2\} \subset \{0,1,2,3\}$ for each $t$. Any set $A_{j_1 i_1} \cap K_c^{-1}A_{j_2 i_2} \cap \cdots \cap K_c^{-(n-1)}(A_{j_n i_n})$ is a product set. Denote its projection on the second coordinate by $I_{j_1 i_1 \ldots j_n i_n}$. Define the set
\[\{t_1, \ldots, t_m\} = \{t : (j_t, i_t) \in \mathcal{A}_c\},\]
where we assume that $1 \le t_1 < t_2 < \cdots < t_m \le n$. These are the indices $t$ such that $j_t=k_t$ and thus the projection of $A_{j_ti_t}$ to the first coordinate does not equal $\{0,1\}^\mathbb N$. This implies that we can write 
\[ A_{j_1 i_1} \cap K_c^{-1} (A_{j_2 i_2} )\cap \cdots \cap K_c^{-(n-1)}(A_{j_n i_n}) = [k_{t_1},k_{t_2}, \ldots ,k_{t_m}] \times I_{j_1 i_1 \ldots j_n i_n}\]
and
\[ m_p \times \mu_{p,c} (h^{-1} ([(k_1,i_1),(k_1,i_2),\ldots ,(k_n,i_n)])) = \sum_{j_1, \ldots, j_n} m_p \times \mu_{p,c} ([k_{t_1},k_{t_2} ,\ldots ,k_{t_m}] \times I_{j_1 i_1 \ldots j_n i_n}).\]

\vskip .2cm
To compute $\tilde \nu (h^{-1} ([(k_1,i_1),(k_1,i_2),\ldots ,(k_n,i_n)]))$, let $\mathcal S \subseteq \{0,1\}^n$ denote the set of blocks $s_1, s_2, \ldots , s_n $ for which $s_t = k_t$ for all $t \in \{ t_1, \ldots, t_m\}$. Then
\[\begin{split}
 \tilde \nu (h^{-1} ([(k_1,i_1),(k_1,i_2), & \ldots ,(k_n,i_n)]))\\
 =\ & \nu (h (Z_K \cap h^{-1} ([(k_1,i_1),(k_1,i_2),\ldots ,(k_n,i_n)])))\\
=\ & \nu (D \cap [(k_1,i_1),(k_1,i_2),\ldots ,(k_n,i_n)])\\
=\ & \nu ([(k_1,i_1),(k_1,i_2),\ldots ,(k_n,i_n)])\\
=\ & m_p \times \mu_{p,c} \bigg( \bigcup_{j_1, \ldots, j_n} \bigcup_{s_1,s_2, \ldots, s_n \in \mathcal S} [s_1 , s_2, \ldots, s_n] \times I_{j_1 i_1 \ldots j_n i_n} \bigg)\\
 =\ & \sum_{j_1, \ldots, j_n} \sum_{s_1,s_2, \ldots, s_n \in \mathcal S}m_p \times \mu_{p,c} ([s_1 , s_2, \ldots, s_n] \times I_{j_1 i_1 \ldots j_n i_n})\\
 = \ & \sum_{j_1, \ldots, j_n} m_p \times \mu_{p,c} ([k_{t_1},k_{t_2}, \ldots ,k_{t_m}] \times I_{j_1 i_1 \ldots j_n i_n}). 
 \end{split}\]
Hence, $\tilde \nu =  m_p \times \mu_{p,c}$ and the statement follows.
\end{proof}

\begin{proof}[Proof of Theorem~\ref{t:universal}]
Let $0<p<1$ be given. For each $s \in \{0,1\}$ and $d \in \{2,3, \ldots, \ell\}$ define
%\[ \Delta (s,d) = \begin{cases}
%[0] \times \big[ \frac1d + \frac{c}{d(d-1)}, \frac1{d-1}-\frac{c}{d(d-1)}\big] \cup \{0,1\}^\mathbb N \times \big( \frac1{d-1}-\frac{c}{d(d-1)} , \frac1{d-1} \big), & \text{ if } s=0,\\
%[1] \times \big[ \frac1d + \frac{c}{d(d-1)}, \frac1{d-1}-\frac{c}{d(d-1)}\big] \cup \{0,1\}^\mathbb N \times \big[ \frac1d , \frac1{d} + \frac{c}{d(d-1)} \big) , & \text{ if } s=1, \, d\ge 3,\\
%[1] \times \big[ \frac1d + \frac{c}{d(d-1)}, \frac1{d-1}-\frac{c}{d(d-1)}\big] \cup \{0,1\}^\mathbb N \times \big( \big[ \frac1d , \frac1{d} + \frac{c}{d(d-1)} \big) \cup \{1\} \big)  , & \text{ if } s=1, \, d=2.
%\end{cases}\]
\[ \hat \Delta (s,d) = \begin{cases}
[0] \times [z_d^+, z_{d-1}^-] \cup \{0,1\}^\mathbb N \times (z_{d-1}^-, z_{d-1} ], &  \text{ if } s=0, \, d=2,\\
[0] \times [z_d^+, z_{d-1}^-] \cup \{0,1\}^\mathbb N \times (z_{d-1}^-, z_{d-1} ), &  \text{ if } s=0, \, d \ge 3\\
[1] \times [z_d^+, z_{d-1}^1] \cup \{0,1\}^\mathbb N \times \big[ z_d , z_d^+ \big) , & \text{ if } s=1.
\end{cases}\]

Fix $(s_1,d_1), \ldots, (s_j,d_j) \in \mathcal A_c$. Then the set
\[ E=\hat \Delta(s_1,d_1) \cap K_c^{-1}\hat \Delta(s_2,d_2) \cap \cdots \cap K_c^{-(j-1)} \hat \Delta (s_j,d_j) \subseteq \{0,1\}^\mathbb N \times [c,1]\]
contains precisely those points $(\omega,x)$ for which $s_i(\omega,x)=s_i$ and $d_i(\omega,x)=d_i$ for all $1 \le i\le j$. Since for each $s,d$, $m_p \times \lambda (\hat \Delta(s,d)) = \frac{c}{d(d-1)} + (p-s)(-1)^s \frac{1-2c}{d(d-1)}>0$ and $K_c (\hat \Delta(s,d)) = \{0,1\}^\mathbb N \times [c,1]$, it also follows that $m_p \times \lambda (E) >0$. The map $K_c$ is ergodic with respect to $m_p \times \mu_{p,c}$ by Lemma \ref{l:invergK} and the measures $\mu_{p,c}$ and $\lambda$ are equivalent by Proposition~\ref{p:equivm}. By Birkhoff's Ergodic Theorem it then follows that for $m_p \times \lambda$-a.e.~$(\omega,x)$ the block $(s_1,d_1), \ldots, (s_j,d_j)$ occurs with positive frequency in the sequence $(s_1(K_c^n(\omega,x)),d_1(K_c^n(\omega,x)))_n$. Since there are only countably many blocks $(s_1,d_1), \ldots, (s_j,d_j)$ it follows that the $c$-L\"uroth expansion of $m_p\times \lambda$-a.e.~$(\omega,x)$ is universal. Let 
\[\tilde{Z} :=\{x \in [c,1]\,:\,  \forall \omega \in \{0,1\}^{\mathbb{N}} \, K_c^n(\omega,x) \in \{0,1\}^{\mathbb{N}} \times S \text{ for infinitely many } n \geq 0 \}.\]
Then $\lambda ([c,1] \setminus \tilde{Z}) =0$ by Lemma~\ref{l:allinS}. From Fubini's Theorem we get the existence of a set $B \subseteq \tilde{Z}$ with $\lambda(B)=1-c$ and for each $x \in B$ a set $A_x \subseteq \{0,1\}^\mathbb N$ with $m_p(A_x)=1$ and such that for any $(\omega,x) \in A_x \times \{x\}$ the sequence $((s_1 (K_c^n(\omega,x)) ,d_1 (K_c^n(\omega,x))))_n$ is universal. Since the set $A_x$ has full measure, it contains uncountably many sequences. For any $x \in \tilde Z$ different sequences $\omega$ define different sequences $((s_1(K_c^n(\omega,x))  , d_1(K_c^n(\omega,x)) ))_{n \geq 1}$. Hence, we obtain for Lebesgue almost every $x$ uncountably many universal $c$-L\"uroth expansions.
\end{proof}

Theorem~\ref{t:main3} is now given by Theorem~\ref{t:uncountable25} and Theorem~\ref{t:universal}.

\section{More examples}
Explicit expressions for the probability density functions $f_{p,c} = \frac{\de \mu_{p,c}}{\de \lambda}$ can be obtained from the procedure from \cite[Theorem 4.1]{KM} in case $p\neq \frac12$ (since otherwise condition (A5) is violated). From this result it follows that
\begin{equation}\label{q:denskopf}
f_{p,c} = c_1 + c_2 \sum_{t \ge 0} \sum_{\omega \in \{0,1\}^t} \frac{p_\omega}{T_{\omega,c}'(1-c)} 1_{[c, T_{\omega,c}(1-c))} + c_3 \sum_{t \ge 1} \sum_{\omega \in \{0,1\}^t} \frac{p_\omega}{T_{\omega,c}'(c)} 1_{[c, T_{\omega,c}(c))},
\end{equation}
where $c_1,c_2,c_3$ are constants and $p_\omega$ is an abbreviation for the product $p_{\omega_1} \cdots p_{\omega_t}$. The sums in this expression have finitely many terms if the random orbits $T_{\omega,c} (1-c)$ and $T_{\omega,c}(c)$ take values in a finite set. This happens for example if the  
%For $c>0$ it is far from trivial to find an expression for the density of the stationary measure $\mu_{p,c}$. However, \cite[Theorem 4.1]{KM} applies to any random map $L_{c,p}$ and can be used to obtain such expressions. In particular, the algebraic procedure of \cite[Theorem 4.1]{KM} is quite straightforward when applied to
random map $L_{c,p}$ is Markov, which is the case for any $c \in \mathbb{Q} \cap (0, \frac12] $ as the next Proposition shows. 

\begin{proposition}\label{p:markov}
For any $c \in \mathbb{Q} \cap (0, \frac12] $ and any $0  < p < 1$ the random $c$-L\"uroth transformation $L_{c,p}$ is Markov.
\end{proposition}

\begin{proof}
%Recall from \eqref{q:zn} that $z_n = \frac{1}{n}$, 
%\[z_n^+ = z_n + c z_n z_{n-1}  \quad \text{ and } \quad z_{n}^- = z_{n} - c z_n z_{n+1}.\]
Let $\mathcal{S}_c= \{s_i\}_i$ be the finite set of points given by 
\[\{c,1\} \cup \{ z_n, z_n^+, z_{n-1}^-\}_{n\ge 2} \cap [c,1],\] 
such that $s_0 = c < s_1 < \ldots < s_k= 1$. These are the critical points in $[c,1]$ of $L_{c,p}$. For $j=0,1$, 
\[T_{j,c}(s_i, s_{i+1}) \in \{(1-c, 1), (c, 1-c), (T_{j,c}(c), 1), (c, T_{j,c}(c)), (T_{j,c}(c),1-c), (1-c, T_{j,c}(c))\},\] 
so that, to determine a Markov partition, it is enough to study the orbit of $c$ and $1-c$. Since $c\in \mathbb{Q} $, Proposition \ref{p:returningdense} implies that the set 
\begin{equation}\label{q:orbitc}
\mathcal{O}_c = \{T_{\omega,c}^n(c) \, : \, \omega \in \Omega^{\mathbb{N}}, \, n \in \mathbb{N} \} \cup \{T_{\omega,c}^n(1-c) \, : \, \omega \in \Omega^{\mathbb{N}}, \, n \in \mathbb{N} \},
\end{equation}
is finite. By construction, the partition obtained by the points in $\mathcal{S}_c \cup \mathcal{O}_c$ is Markov. 
\end{proof}

\begin{remark}
From the previous proposition it follows that for each rational $c \in \big(0, \frac12 \big)$ the density $f_{p,c}$ is given by a finite sum of weighted indicator functions. In that case the weights can be explicitly computed by solving the homogeneous matrix equation from the procedure from \cite{KM}. With a proof very similar to the one for \cite[Theorem 4.1]{DK} it can be shown that for $c \neq z_n, z_n^+, z_n^-$ the densities $f_{p,c_k} \to f_{p,c}$ in $L^1(\lambda)$ if $c_k \to c$, $c_k, c \in (0, \frac12)$. If $c$ equals one of $ z_n, z_n^+, z_{n-1}^-$ for some $n \ge 3$, then this statement still holds for a sequence $(c_k)_k$ converging to $c$ from the right. Hence, for irrational $c \in (0, \frac12)$ one can approximate $f_{p,c}$ with expressions as in \eqref{q:denskopf}.
\end{remark}

From Proposition~\ref{p:markov} we can immediately determine a Markov partition for $L_{c,p}$ in case $c= \frac{1}{2^k}$. 

\begin{corollary}\label{c:partition2k}
For any integer $k \geq 1$, the Markov partition of $L_c$ for $c=\frac{1}{2^k}$ is given by
\[\mathcal{S}_{\frac{1}{2^k}} \ \cup \ \bigcup_{i=2}^{k} \bigg\{1- \frac{1}{2^i} \bigg\} .\] 
\end{corollary}

\begin{proof}
First, note that for any integer $\ell >3 $, the map $T_{1, \frac1\ell}$ is Markov, and its Markov partition can be given by
\[ \mathcal{S}_{\frac1\ell} \ \cup \ \bigg\{ 1-\frac1\ell, \frac2\ell \bigg\}. \]
Indeed, since $\frac12 < 1-\frac1\ell < 1-\frac1{2\ell}$, then $T_{1, \frac1\ell} (1-\frac1\ell)= \frac2\ell$. If $\ell$ is even, the point $z_{\frac{\ell}{2}} = \frac2\ell$ is already in $\mathcal{S}_{\frac1\ell}$, otherwise $\frac{2}{\ell+1} < \frac2\ell < \frac{2}{\ell-1}$ and moreover $\frac2\ell  < z_{\frac{\ell-1}{2}}^- = \frac{2}{\ell-1} - \frac2{\ell-1}\frac2{\ell+1} \frac1\ell$. Thus, $T_{1, \frac1\ell} \big(\frac2\ell\big)= \frac12 + \frac1{2\ell} \in \mathcal{S}_{\frac1\ell}$.\\

We now follow the random orbit of $1-c=1-\frac{1}{2^k}$. Since
\[z_2^+=\frac{1}{2}+\frac{1}{2^{k+1}} < 1-\frac{1}{2^i} < 1- \frac{1}{2^{k+1}}=z_1^-,\]
for $i= 2, 3, \ldots, k$ then 
\[\mathcal{O}_{\frac{1}{2^k}} = \{T_{\omega_1^n}(1-c) \, : \, \omega \in \Omega^{\mathbb{N}}, n \in \mathbb{N}\} = \bigcup_{i=2}^{k} \bigg\{1-\frac{1}{2^i}, \frac{1}{2^{i-1}},1\bigg\}. \qedhere\]
\end{proof} 

In the following, we give examples for $k=2$ and $k=3$.

\begin{example}
For $k=2$ a Markov partition of $L_{\frac14}$ is given by the points
\[\bigg\{\frac{1}{4}, \frac{13}{48}, \frac{11}{36}, \frac{1}{3}, \frac{3}{8}, \frac{11}{24}, \frac{1}{2}, \frac{5}{8}, \frac{3}{4}, \frac{7}{8}, 1 \bigg\}.\]
The application of \cite[Theorem 4.1]{KM} yields the following expression for the normalised density function
\begin{equation}\label{density2}
f_{p,\frac14}(x)= \begin{cases}
\dfrac{4}{2p+3} & \mbox{if  } x \in \displaystyle \bigg[\frac14, \frac12\bigg),\\[10pt]
\dfrac{4p+2}{2p+3} & \mbox{if  } x \in \displaystyle \bigg[\frac12, \frac34\bigg),\\[10pt]
2 & \mbox{if  } x \in \displaystyle \bigg[\frac34, 1\bigg].
\end{cases}
\end{equation}
The explicit formula for the unique invariant density allows to say more on the digit frequency and the Lyapunov exponent. Recall from Example \ref{e:L13more} that the frequency of a digit $d \in \{2,3,4\}$ is given by
\[\pi_d=  \lim_{n \to \infty} \frac{ \# \{ 1 \leq j \leq n: \,  d_j(\omega, x)=d\}}{n} = \pi_{(0,d)}+ \pi_{(1,d)}.\]
It then follows by Birkhoff's Ergodic Theorem that, for $m_p \times \mu_{p,\frac14 }$-a.e. point $(\omega, x)$,
\[\pi_2= \frac{2p+2}{2p+3}, \quad \pi_3= \frac{2}{3(2p+3)}, \quad \pi_4= \frac{1}{3(2p+3)}.\]
Recall the definition of $\Lambda$ from \eqref{d:Lyap}. We obtain 
\[\Lambda_{m_p \times \mu_{p,\frac14}} = \frac{p \log 64 + \log 27648}{6p+9} < \Lambda_{m_p \times \lambda}.\]
That is, for $m_p \times \mu_{ p,\frac14}$-a.e. point, the approximants $\frac{p_n}{q_n}$ obtained by the iteration of the random $\frac14$-L\"uroth map are in general worse than the corresponding ones obtained via the random $0$-L\"uroth map with countably many branches. 
\end{example}

\begin{example}
For $k=3$ the $L_{\frac18}$-invariant probability density of the measure $\mu_{p,\frac18}$ is
\begin{equation}\nonumber
f_{p,\frac18}(x)= \begin{cases}
\dfrac{8}{2p^2+3p+5} & \text{if  } x \in \displaystyle \bigg[\frac18, \frac14\bigg), \\[10pt]
\dfrac{4p+4}{2p^2+3p+5} & \text{if  } x \in \displaystyle \bigg[\frac14, \frac12\bigg),\\[10pt]
\dfrac{4p^2+2p+4}{2p^2+3p+5} & \text{if  } x \in \displaystyle \bigg[\frac12, \frac34\bigg),\\[10pt]
\dfrac{4p^2+6p+4}{2p^2+3p+5} & \text{if  } x \in \displaystyle \bigg[\frac34, \frac78\bigg),\\[10pt]
\dfrac{4p^2+6p+12}{2p^2+3p+5} & \text{if  } x \in \displaystyle \bigg[\frac78,1\bigg].
\end{cases}
\end{equation}
The frequency of the digits $d \in \{2,3, \ldots, 8\}$ is given by 
\[\begin{array}{lll}
\pi_2  = \frac{2p^2+2p+3}{2p^2+3p+5}, & \pi_3= \frac{2(p+1)}{3(2p^2+3p+5)}, &\pi_4= \frac{p+1}{3(2p^2+3p+5)}, \\[5pt]
\pi_5  = \frac{2}{5(2p^2+3p+5)} ,& \pi_6= \frac{4}{15(2p^2+3p+5)}, & \pi_7= \frac{2}{21(2p^2+3p+5)}, \\[5pt]
\pi_8  = \frac{1}{7(2p^2+3p+5)}. & &
\end{array}\]
Moreover for $m_p \times \mu_{p,\frac18}$-a.e. $(\omega, x)$ we have by \eqref{d:specificlyap} that 
\begin{equation}\nonumber
\begin{split}
\Lambda(\omega, x) & = \sum_{d=2}^{8} \log (d(d-1)) \mu_{p,\frac18} \bigg( \bigg[ \frac{1}{d}, \frac{1}{d-1} \bigg)\bigg) \\
& = \dfrac{8}{2p^2+3p+5} \bigg( \frac{ \log 56}{56} + \frac{ \log 42}{42} + \frac{ \log 30}{30} + \frac{ \log 20}{20} \bigg) \\
& \quad  + \dfrac{4p+4}{2p^2+3p+5} \bigg(\frac{ \log 12}{12} + \frac{ \log 6}{6}  \bigg) + \dfrac{2p^2+2p+3}{2p^2+3p+5} \log 2 \\
& \cong \dfrac{1.38628 p^2 +  3.40908 p+ 7.49448}{2p^2+3p+5} < \frac32< \Lambda_{m_p \times \lambda}.
\end{split}
\end{equation}
\end{example}

\bibliographystyle{alpha}
\bibliography{random}

\end{document}